\documentclass[12pt]{amsart}

\usepackage[utf8]{inputenc}
\usepackage[OT2,T1]{fontenc}
\usepackage{ae}
\usepackage{aecompl}

\usepackage{amsmath,amsthm,amssymb,mathrsfs,calligra,dsfont,mathtools}
\usepackage[final, babel=true]{microtype}
\usepackage{caption}
\usepackage{subcaption}
\usepackage{eucal}
\usepackage{verbatim}
\usepackage{centernot}
\usepackage{enumitem}
\usepackage{graphicx}
\usepackage{tikz}
\usetikzlibrary{matrix,arrows,cd}
\usepackage[all]{xy}
\usepackage{hyperref}
\usepackage{xcolor}
\usepackage{url}
\definecolor{darkgreen}{rgb}{0,0.45,0}

\hypersetup{
	colorlinks,
	linktoc=all,
	linkcolor={darkgreen},
	citecolor={blue!50!black},
	urlcolor={blue!80!black}
}

\DeclareFontFamily{U}{mathx}{\hyphenchar\font45}
\DeclareFontShape{U}{mathx}{m}{n}{
	<5> <6> <7> <8> <9> <10>
	<10.95> <12> <14.4> <17.28> <20.74> <24.88>
	mathx10
}{}
\DeclareSymbolFont{mathx}{U}{mathx}{m}{n}
\DeclareFontSubstitution{U}{mathx}{m}{n}
\DeclareMathAccent{\widecheck}{0}{mathx}{"71}
\DeclareMathAccent{\wideparen}{0}{mathx}{"75}

\newcommand\bA{\mathbb{A}}

\newcommand\bI{\mathbb{I}}
\newcommand\bR{\mathbb{R}}
\newcommand\bC{\mathbb{C}}

\newcommand\bZ{\mathbb{Z}}
\newcommand\bQ{\mathbb{Q}}
\newcommand\bF{\mathbb{F}}

\newcommand\bT{\mathbb{T}}

\newcommand\Zp{{\mathbb{Z}_p}}
\newcommand\Qp{{\mathbb{Q}_p}}

\newcommand\cE{\mathcal{E}}

\newcommand\cH{\mathcal{H}}

\newcommand\cK{\mathcal{K}}
\newcommand\cL{\mathcal{L}}
\newcommand\cM{\mathcal{M}}

\newcommand\cO{\mathcal{O}}

\newcommand\cT{\mathcal{T}}

\newcommand\cX{\mathcal{X}}

\newcommand\ga{\mathfrak{a}}

\newcommand\gh{\mathfrak{h}}

\newcommand\gm{\mathfrak{m}}

\newcommand\sH{\mathscr{H}}

\newcommand{\ob}[1]{\mkern 1.5mu\overline{\mkern-1.5mu#1\mkern-1.5mu}\mkern 1.5mu}

\newcommand{\Ind}{\mathrm{Ind}}

\renewcommand{\char}{\textrm{char}}

\DeclareMathOperator{\Gal}{Gal}
\DeclareMathOperator{\Der}{Der}

\DeclareMathOperator{\id}{id}
\DeclareMathOperator{\Hom}{Hom}

\DeclareMathOperator{\im}{im}

\DeclareMathOperator{\coker}{coker}
\DeclareMathOperator{\Frac}{Frac}

\DeclareMathOperator{\ad}{ad}
\DeclareMathOperator{\tr}{tr}
\DeclareMathOperator{\End}{End}

\DeclareMathOperator{\cond}{cond}

\DeclareMathOperator{\GL}{GL}

\DeclareMathOperator{\Sel}{Sel}
\DeclareMathOperator{\Gr}{Gr}
\DeclareMathOperator{\diff}{d}

\DeclareMathOperator{\Spec}{Spec}
\DeclareMathOperator{\corank}{corank}

\newcommand{\HH}{\mathrm{H}}

\newcommand{\ur}{\mathrm{ur}}

\DeclareSymbolFont{cyrletters}{OT2}{wncyr}{m}{n}
\DeclareMathSymbol{\Sha}{\mathalpha}{cyrletters}{"58}

\makeatletter
\@namedef{subjclassname@2020}{\textup{2020} Mathematics Subject Classification}
\makeatother

\theoremstyle{defstyle}
\newtheorem{definition}[subsubsection]{Definition}

\newtheorem{remark}[subsubsection]{Remark}

\theoremstyle{thmstyle}
\newtheorem{theorem}[subsubsection]{Theorem}
\newtheorem{proposition}[subsubsection]{Proposition}
\newtheorem{lemma}[subsubsection]{Lemma}

\newtheorem*{conjecture*}{Conjecture}

\newtheorem{lemmaAppendix}{Lemma}
\newtheorem{propAppendix}{Proposition}

\newtheorem{corAppendix}{Corollary}

\newtheorem{THM}{Theorem}

\tikzset{
	symbol/.style={
		draw=none,
		every to/.append style={
			edge node={node [sloped, allow upside down, auto=false]{$#1$}}}
	}
}

\author{Alexandre Maksoud}
\address{Universität Paderborn, Warburger Str. 100, 33098 Paderborn, Germany}
\email{maksoud.alexandre@gmail.com}
\urladdr{https://sites.google.com/view/alexandre-maksoud/}

\title{A canonical generator for congruence ideals of Hida families}
\begin{document}
\maketitle
\begin{abstract}
	We construct canonical adjoint $p$-adic $L$-functions generating the congruence ideal attached to Hida families using Ohta's pairing. We show that these $p$-adic $L$-functions, suitably modified by certain Euler factors, are interpolated by a regular element of Hida's universal ordinary Hecke algebra. We also relate them to characteristic series of primitive adjoint Selmer groups.
\end{abstract}
\section{Introduction}
In the context of $L$-functions, the notion of period often refers to Deligne's period for the value of a motivic $L$-function at a critical integer. When the $L$-function has an automorphic origin, its special values can sometimes be related to automorphic periods. Consider an elliptic cuspidal newform $f$ of weight $k\geq 2$ and level $N$, and let $L(\ad f,s)$ be the $L$-function attached to the adjoint lift of the automorphic representation of $\GL_2(\bA)$ generated by $f$. A theorem of Sturm \cite{sturm1980special} implies that the ratio
\begin{equation}\label{eq:intro_symm_square}
	\dfrac{L(\ad f,j)}{\pi^{2j+k-1}\langle f,f\rangle}
\end{equation}
is an algebraic number for all odd integers $1\leq j \leq k-1$, where $\langle f,f\rangle$ is the Petersson norm of $f$ for $\Gamma_1(N)$. The same automorphic period $\langle f,f\rangle$ is involved in algebraicity statements on special values of the Rankin-Selberg $L$-function $L(f\times g,s)$, where $g$ is another newform of weight smaller than $k$.

Let $p\geq 5$ be a prime number. When $f$ is $p$-ordinary, Schmidt \cite{schmidt} constructed a $p$-adic $L$-function $\cL_p(\ad f,s)$ interpolating (up to simple fudge factors) quantities like \eqref{eq:intro_symm_square} or twisted variants thereof. Hida showed in \cite{hidaGL3} that this construction could be realized for the primitive Hida family $F$ of $f$, yielding a two-variable $p$-adic $L$-function $\cL_p(\ad F,\kappa,s)$, $\kappa$ being the weight variable. 
However, Schmidt-Hida's construction does not quite fit into Coates-Perrin-Riou's framework of $p$-adic $L$-functions of motives developed in \cite{coatesperrinriou}. Indeed, the motivic periods for $\ad f$ and its critical twists should involve the product $\Omega_f^+\Omega_f^-$ of $p$-normalized Shimura periods of $f$ rather than $\langle f,f \rangle$. 
This problem of normalizations of periods is perhaps better directly seen on \eqref{eq:intro_symm_square} for $j=1$, which turns out to be an explicit number depending on $f$ in an elementary way and does not seem to capture deep algebraic invariants such as the size of a Bloch-Kato Selmer group.
This observation suggests that $\cL_p(\ad F,\kappa,s)$ is the ratio of two ''genuine`` (or ''motivic``) $p$-adic $L$-functions $L_p(\ad F,\kappa,s)$ and $L_p(\ad F,\kappa)$ interpolating (up to simple fudge factors)
	\begin{equation}\label{eq:intro_ratio_of_periods}
		\dfrac{L(\ad f,j)}{\pi^{2j+k-1}\Omega_f^+\Omega_f^-} \quad \mbox{and} \quad	\dfrac{\langle f,f\rangle}{\Omega_f^+\Omega_f^-}
	\end{equation}
respectively.
The second one should also generate the congruence ideal of $F$ in most cases. By Hida's formula mentioned above, $\pi^{k+1}\langle f,f\rangle$ is essentially equal to $L(\ad f,1)$, and it is therefore reasonable to call $L_p(\ad F,\kappa)$ the (weight-variable) adjoint $p$-adic $L$-function of $F$. 
These expectations are not new; a similar picture for Rankin-Selberg $L$-series and their $p$-adic counterparts is drawn by Hida in his monograph \cite{hidagenuine} where many evidences are provided in the CM case.

This article focuses on the construction of $L_p(\ad F,\kappa)$ and the study of its arithmetic properties. 
Hida  \cite{hidamodules} gave first evidences for the existence of $L_p(\ad F,\kappa)$ by showing in certain cases that any generator of the congruence ideal of $F$, evaluated at $f$, is essentially equal to the second ratio of \eqref{eq:intro_ratio_of_periods} (see also \cite[Corollary 6.29]{hidapune}).
A construction of an adjoint $p$-adic $L$-function (or, rather, an invertible sheaf) on the Eigencurve of tame level $N$ was carried out by Bella\"iche \cite{eigenbook} using Kim's scalar product on overconvergent modular symbols.
Its vanishing locus is intimately related to the ramification locus of the weight map, which constitutes a geometric counterpart to the support of the congruence module.
Note that Bella\"iche's interpolation formula is only given at crystalline points of trivial Nebentype. Locally on the Eigencurve, the Shimura periods can be normalized in a coherent way at all arithmetic points, with a controllable $p$-adic error term. This approach has been generalized to other settings in recent works by Balasubramanyam-Longo, Lee, Wu and Lee-Wu \cite{baskarlongo,lee2021padic,wusiegel,leewu}.

Going back to the ordinary setting, we now set up some notations and we state our main results. 

Let $\ob{\bQ}_p$ be an algebraic closure of $\Qp$ and fix an isomorphism $\bC\simeq \ob{\bQ}_p$. Let $\Lambda=\Zp[[1+p\Zp]]$ be the Iwasawa algebra. Assume $F$ is a primitive Hida family of tame level $N$ with coefficients in an integrally closed finite flat $\Lambda$-algebra $\bI$. Denote by $\ob{\rho}_F$ residual $p$-ordinary Galois representation over $G_{\bQ}=\Gal(\ob{\bQ}/\bQ)$ attached to $F$ and assume the following condition:
\begin{equation} \tag{CR}\label{tag:cr}
	\ob{\rho}_F \textrm{ is absolutely irreducible and $p$-distinguished.}
\end{equation}
Here, $\ob{\rho}_F$ is called $p$-distinguished if the semi-simplification of its restriction to $G_{\bQ_p}=\Gal(\ob{\bQ}_p/\Qp)$ is non-scalar.

\begin{THM}[=Theorem \ref{thm:calcul_valeurs_speciales}]\label{thm:introduction_main_1}
	Assume $\ob{\rho}_F$ satisfies \eqref{tag:cr}. There exists a generator $L_p(\ad F)\in\bI$ of the congruence ideal of $F$ satisfying the following interpolation property. Let $f$ be an arithmetic specialization of $F$ of weight $k\geq 2$ and level $Np^r$, $r\geq 1$, and denote by $f^\circ$ its associated newform. Then the specialization $L_p(\ad f)$ of $L_p(\ad F)$ at $f$ is given by
	\begin{equation*}
		L_p(\ad f)=p^{r-1}\alpha^r w(f^\circ)\cE_p(\ad f^\circ)\cdot (-2i)^{k} \dfrac{\langle f^\circ,f^\circ\rangle}{\Omega^+_{f^\circ}\Omega^-_{f^\circ}}\in\ob{\bQ}_p.
	\end{equation*}
	In the above formula, $\alpha$ and $\beta$ are the roots of the $p$-th Hecke polynomial of $f^\circ$ with $|\alpha|_p=1$, $w(f^\circ)\in \ob{\bQ}^\times$ is the Atkin-Lehner pseudo-eigenvalue of $f^\circ$, and $\cE_p(\ad f^\circ)$ is a modified Euler factor at $p$ given by $\cE_p(\ad f^\circ)=1$ if $f^\circ= f$ and by \[\cE_p(\ad f^\circ)=\alpha(p-1)(1-\frac{\beta}{\alpha})(1-p^{-1}\frac{\beta}{\alpha}) \]
	if $f^\circ\neq f$.
\end{THM}
Let $\ob{\rho}=\ob{\rho}_F$ and let $\bT$ be the local component of Hida's big ordinary Hecke algebra of tame level $N$ associated with $\ob{\rho}$. 
To $\ob{\rho}$ is also attached a piece $e\HH_1(Np^\infty)_{\ob{\rho}}$ of the singular homology of the tower of compactified modular curves $(X_1(Np^r))_{r\geq 1}$, where $e$ denotes Hida's ordinary projector. 
Under \eqref{tag:cr}, the work of Emerton-Pollack-Weston  \cite{emerton2006variation} shows that $e\HH_1(Np^\infty)_{\ob{\rho}}^\pm$ is $\bT$-free of rank 1, where the superscript $\pm$ stands for the $\pm1$-eigenspace for the complex conjugation. 
The choice of a $\bT$-basis $\xi^\pm$ of $e\HH_1(Np^\infty)_{\ob{\rho}}^\pm$ and Hida's control theorem then simultaneously pin down Shimura periods for all classical eigenforms $f$ with residual Galois representation isomorphic to $\ob{\rho}$. Shimura periods for the newform $f^\circ$ associated to $f$ are defined in Remark \ref{rem:periods_associated_newforms}.

In order to construct $L_p(\ad F)$, we first pass, via intersection of cycles, from $e\HH_1(Np^\infty)_{\ob{\rho}}^\pm$ to a cohomology space $e^*\HH^1(Np^\infty)_{\ob{\rho}}^\pm$ where $e^*$ is an \emph{anti-ordinary} projector. We then use a work of Ohta \cite{ohtacrelle} to obtain a perfect $\Lambda$-adic pairing on  $e^*\HH^1(Np^\infty)_{\ob{\rho}}^+\times  e^*\HH^1(Np^\infty)_{\ob{\rho}}^-$ which interpolates Petersson norms of anti-ordinary eigenforms at finite level. 
Since members of $F$ contribute to ordinary cohomology instead of anti-ordinary cohomology, we make use of the Atkin-Lehner involution, which is known by Hida to
transports integral structures between ordinary and anti-ordinary cohomology (see Proposition \ref{prop:hida}).

Fix a modular $p$-ordinary residual representation $\ob{\rho}\colon G_\bQ\to \GL_2(\ob{\bF}_p)$ satisfying \eqref{tag:cr} and a finite set $\Sigma$ of primes not containing $p$. The $p$-ordinarity means that $\ob{\rho}$ admits a $G_{\Qp}$-unramified one-dimensional quotient which is fixed. These choices give rise to a universal $\Lambda$-adic Hecke algebra $\bT_\Sigma(\ob{\rho})=\bT_\Sigma$ whose spectrum parameterizes Hida families with residual $p$-ordinary Galois representation isomorphic to $\ob{\rho}$ that are minimally ramified outside of $\Sigma$. In the next theorem, we choose $\bI$ big enough so that it contains the coefficients of any Hida family appearing in $\Spec(\bT_\Sigma)$.

\begin{THM}\label{thm:introduction_main_2}
	There exists a regular element $L_\Sigma(\ad \ob{\rho})\in\bT_\Sigma\otimes_\Lambda\bI$ whose specialization at any primitive Hida family $F$ minimally ramified outside $\Sigma$ with $\ob{\rho}_F\simeq\ob{\rho}$ is given by $E_\Sigma(\ad F)L_p(\ad F)$ up to $\bI$-units, where $E_\Sigma(\ad F)\in\bI$ is a product of Euler factors associated with the adjoint of $F$ at certain primes of $\Sigma$. 
\end{THM}
See Theorem \ref{thm:main_2} for a precise statement of this theorem, and in particular, \eqref{eq:definition_N_Sigma} for the definition of the level $N_\Sigma$ of $\bT_\Sigma$ and \eqref{eq:euler_factor} for the definition of the Euler factors. 
Theorems \ref{thm:introduction_main_1} and \ref{thm:introduction_main_2} together can be interpreted as the existence of congruences between Petersson norms of different $p$-ordinary newforms sharing the same residual Galois representation.
As for $L_p(\ad F)$, the element $L_\Sigma(\ad \ob{\rho})$ only depends on the choice of a $\bT_\Sigma$-basis of $e\HH_1(N_\Sigma p^\infty)_{\ob{\rho}}^\pm$.
We use here Fukaya-Kato's variant of Ohta's pairing with values in a space of $\Lambda$-adic cuspforms, which turns out to be $\bT_\Sigma$-free since $\bT_\Sigma$ is Gorenstein.
The appearance of Euler factors comes from the comparison of cohomology at two different levels (as in classical level raising arguments) for which we make use of the main technical result of \cite{emerton2006variation}. 
The author hopes in near future to apply Theorem \ref{thm:introduction_main_2} to the study of the variation of Iwasawa invariants associated with primitive Rankin-Selberg $p$-adic $L$-functions, as constructed in \cite{chenhsieh}. 

We next relate adjoint $p$-adic $L$-functions to their algebraic counterparts called Selmer groups. Let $F$ be a primitive Hida family whose residual Galois representation satisfies \eqref{tag:cr} and let $V_{\ad}=\bI^{\oplus 3}$ be the adjoint Galois representation associated with $F$. By ordinarity, $V_{\ad}$ has a two-dimensional quotient $V_{\ad}^-$ which is stable under the action of $G_{\bQ_p}$. We let $\bI^\vee=\Hom(\bI,\Qp/\Zp)$ be the Pontryagin dual of $\bI$, and set $D_{\ad}=V_{\ad}\otimes_{\bI} \bI^\vee$, $D_{\ad}^-=V_{\ad}^-\otimes_{\bI} \bI^\vee$. The (primitive) adjoint Selmer group $\Sel(\ad F)$ of $F$ is defined as the kernel of the global-to-local restriction map
\[\HH^1(\bQ,D_{\ad}) \longrightarrow \prod_{\ell\neq p} \HH^1(I_\ell,D_{\ad}) \times \HH^1(I_p,D^-_{\ad}),\]
where $I_\ell$ denotes the inertia subgroup at a prime $\ell$. Its Pontryagin dual $\Sel(\ad F)^\vee$ is a finitely generated $\bI$-module. Since $\bI$ is integrally closed, it makes sense to consider the characteristic ideal of a finitely generated torsion $\bI$-module $M$ which we denote by $\char_\bI M$. 

In the next theorem, we make a slightly stronger assumption than \eqref{tag:cr} on residual Galois representations:
\begin{equation}\tag{CR'}\label{tag:cr'}
	\textrm{$\ob{\rho}$ is $p$-distinguished, and $\ob{\rho}_{|G_{\bQ(\mu_p)}}$ is absolutely irreducible.}
\end{equation}
We also need a certain finite set of exceptional primes, defined as
\begin{equation}\label{eq:definition_exceptional_primes}
	\Sigma_e:=\left\{\ell \colon \pi_{f,\ell}\textrm{ is supercuspidal and } \pi_{f,\ell}\simeq\pi_{f,\ell}\otimes \tau_{\bQ_\ell^2}\right\},
\end{equation}
where $\pi_{f,\ell}$ is the automorphic representation of $\GL_2(\bQ_\ell)$ of any arithmetic specialization $f$ of $F$ and $\tau_{\bQ_\ell^2}$ is the unramified quadratic character of $\bQ_\ell^\times$. The definition of $\Sigma_e$ does not depend on the choice of $f$.
\begin{THM}\label{thm:introduction_main_3}
	Assume that $\ob{\rho}_F$ satisfies \eqref{tag:cr'}. The Selmer group $\Sel(\ad F)^\vee$ is a torsion $\bI$-module. If $F$ is twist-minimal, then the equality
	\begin{equation*}
		\char_\bI\Sel(\ad F)^\vee=\prod_{\ell\in\Sigma_e}(1+\ell^{-1})^{-1} L_p(\ad F).\bI
	\end{equation*}
	holds.
\end{THM}
Recall that $F$ is twist-minimal if it has minimal level among its twists. 
Note that $\Sel(\ad F)$ is invariant under twists, but $L_p(\ad F)$ isn't, so twist-minimality is a necessary hypothesis in Theorem \ref{thm:introduction_main_3}. 
However, the relation between adjoint $p$-adic $L$-functions of Hida families and their twists is well-understood (see Proposition \ref{prop:twists_and_adjoint_p_adic_L}), so this is not a restricting hypothesis. We note that a version of Theorem \ref{thm:introduction_main_3} is stated in \cite[Coro. 6.29]{hidapune} under rather restrictive assumptions, forcing $F$ to be minimally ramified.

Under assumption \eqref{tag:cr'}, the ring $\bT_\Sigma$ turns out to be the universal ring $R_\Sigma$ for $p$-ordinary Galois deformations of $\ob{\rho}$ that are minimally ramified outside $\Sigma$. 
As noticed by Mazur and Tilouine \cite{mazurtilouine}, the module of continuous Kähler differentials of $R_\Sigma$ as a $\Lambda$-algebra is related to both adjoint Selmer groups and congruence ideals of systems of eigenvalues appearing in $\bT_\Sigma$. 
This observation is a crucial input in works of Wiles, Taylor-Wiles, and Diamond-Flach-Guo among others and has applications to modularity and to the Bloch-Kato conjecture \cite{wilesfermat,taylorwiles,DFG}. 
We obtain a relation between $\Sigma$-imprimitive Selmer groups and $p$-adic $L$-functions, from which we deduce the equality in Theorem \ref{thm:introduction_main_3} via a generalization of Greenberg-Vatsal's argument \cite{greenbergvatsal2000}.

The plan of this article is as follows. In Section \ref{sec:section_2_construction_p_adic_L} we give the construction of an adjoint $p$-adic $L$-function as in Theorem \ref{thm:introduction_main_1} and we introduce $p$-adic $L$-functions with values in big Hecke algebras. We analyze in Section \ref{sec:section_3_main_conjecture} how these objects change when we vary the tame level, yielding the proof of Theorem \ref{thm:introduction_main_2}. We then study Selmer groups and prove Theorem \ref{thm:introduction_main_3} in the end of Section \ref{sec:section_3_main_conjecture}. Appendix \ref{appendix} recalls definitions and standard facts related to congruence modules and congruence ideals.

\subsection*{Acknowledgments}
The author would like to thank Ming-Lun Hsieh for his comments on this work. This research is partially supported by the Luxembourg National Research Fund, Luxembourg, INTER/ANR/ 18/12589973 GALF.

\subsection*{Notations}
We adopt the convention that the usual Hecke operators (acting on modular forms, cohomology, etc.) in level $M$ are denoted $T_n$, even in the case when $n$ is not coprime to $M$. In particular, for a prime $\ell$ dividing $M$, $T_\ell$ is Atkin-Lehner's operator $U_\ell$. We use the letter $\gh$ for Hecke algebras, and $\bT$ for local components of Hecke algebras (\textit{i.e.} their localizations at maximal ideals). Adjoint Hecke operators for the Petersson scalar product and adjoint Hecke algebras are typically denoted with an asterisk ($T_n^*$, $\gh^*$, $\bT^*$, etc.).

\section{Construction of a $p$-adic $L$-function}\label{sec:section_2_construction_p_adic_L}
\subsection{Adjoint $L$-values}\label{sec:adjoint_L_values}
We fix in this paragraph an integer $M\geq 3$. For $\Gamma=\Gamma_1(M)$, we let $S_k(M)$ be the space of modular cusp forms of weight $k\geq 1$ and level $\Gamma$. Recall the usual slash operator in weight $k$ defined by
\[(f|s)(z)= \det(s)^{k-1}(cz+d)^{-k}f(\frac{az+b}{cz+d}) \] 
for every function $f\colon \sH \to \bC$ on the upper half-plane and every $s=\left(\begin{smallmatrix}a & b \\ c & d\end{smallmatrix}\right)\in S^+=M_2(\bZ)\cap \GL_2(\bQ)^+$.

The abstract Hecke ring $\cH(S^+,\Gamma)$ consisting of double cosets $[\Gamma s \Gamma]$, $s\in S^+$ acts on $S_k(M)$. For $n,d\geq 1$ with $(d,M)=1$, we denote by $T_n$, $T^*_n$, $\langle d \rangle$ the usual Hecke operators, adjoint Hecke operators and diamond operators respectively. The Atkin-Lehner involution is $W_M=[\Gamma\tau_M\Gamma]$, where
\[\tau_M=\begin{pmatrix}0 & -1 \\ M & 0\end{pmatrix}.\]
Note that $\tau_M$ normalizes $\Gamma$, and $T_n W_M=W_M T^*_n$, $T^*_nW_M=W_MT_n$ and $\langle d \rangle W_M=W_M\langle d\rangle^{-1}$.

The Petersson scalar product is defined by the formula
\[\langle f,g \rangle_\Gamma = \int_{X(\bC)}f(z)\ob{g(z)}y^{k-2}dxdy \qquad (z=x+iy),\]
 where $f,g\in S_k(M)$ and $X=X_1(M)_{/\bC}$ is the compactified modular curve of level $\Gamma$. The adjoint of $[\Gamma s \Gamma]$ for $\langle \cdot,\cdot \rangle_\Gamma$ is $[\Gamma s' \Gamma]$, where $s'=\det(s)s^{-1}$, \textit{i.e.},
 \[s'=\begin{pmatrix}d & -b \\ -c & a\end{pmatrix} \qquad \mbox{if } s=\begin{pmatrix}a & b \\ c & d\end{pmatrix}.\]
  The operators $T_n$ and $T_n^*$ (resp. $\langle d\rangle$, and $\langle d\rangle^{-1}$) are adjoint to each other under $(\cdot,\cdot)_\Gamma$.
 
 The Nebetype action decomposes $S_k(M)$ as a direct sum $\oplus S_k(M,\chi)$ indexed on the set of Dirichlet characters $\chi \colon (\bZ/M\bZ)^\times \to \bC^\times$ and which is Hecke-compatible. 
  
 Let $f(q)=\sum_{n\geq 1} a_n q^n\in S_k(M,\chi)$ be a normalized eigenform (\textit{i.e.}, $a_1=1$ and $f|T_n=a_nf$ for all $n\geq 1$). For every prime $\ell$, let $\alpha_\ell$ and $\beta_\ell$ be such that 
 \[(1-\alpha_\ell X)(1-\beta_\ell X)=1-a_\ell X + \chi(\ell)\ell^{k-1} X^2,\] 
where $\chi(\ell)=0$ if $\ell |M$. Define 
\small\[L^\text{imp}(s,\ad f)=\prod_\ell  (1-\frac{\alpha_\ell^2}{\chi(\ell)\ell^{k-1}}\ell^{-s})(1-\frac{\alpha_\ell\beta_\ell}{\chi(\ell)\ell^{k-1}}\ell^{-s})(1-\frac{\beta_\ell^2}{\chi(\ell)\ell^{k-1}}\ell^{-s}),\]\normalsize
the product being taken over all prime numbers $\ell$. The product converges for $\Re(s)>>0$ and it has a meromorphic continuation to $\bC$.  When $f$ is new, $L^\textrm{imp}(s,\ad f)$ coincides, up to Euler factors at certain primes $\ell|M$, with the $L$-function $L(\ad f,s)$ attached to the adjoint lift of the automorphic representation of $\GL_2(\bA)$ generated by $f$.
\begin{proposition}\label{prop:formula_Hida_imprimitive_L}
	Assume $f$ is a newform of level $M$ and Nebentype $\chi$, and let $M_\chi$ be the conductor of $\chi$. Then
	\[L^\text{imp}(1,\ad f)= \frac{2^{2k}\pi^{k+1}}{(k-1)!MM_\chi\varphi(M/M_\chi)} \langle f,f \rangle_\Gamma,\]
	where $\varphi$ is Euler's totient function.
\end{proposition}
\begin{proof}
	This follows from \cite[Thm. 5.1]{hida1981inventiones}, using that $M\geq 3$ by assumption.
\end{proof}
\begin{remark}
	Assume $f$ is twist-minimal. It is perhaps nicer to reformulate Proposition \ref{prop:formula_Hida_imprimitive_L} in terms of primitive adjoint $L$-functions as
	\begin{align*}
		\langle f,f \rangle_\Gamma &= \dfrac{(k-1)!M\varphi(M)}{2^{2k}\pi^{k+1}} \prod_{\ell\in\Sigma_e}(1+\ell^{-1}) L(\ad f,1) \\
							&=M\varphi(M)2^{1-k} \prod_{\ell\in\Sigma_e}(1+\ell^{-1}) \Gamma(\ad f,1) L(\ad f,1),
	\end{align*}
	where $\Sigma_e$ is the set of exceptional primes for $f$ of the introduction and $\Gamma(\ad f,s)=\Gamma_\bC(s+k-1)\Gamma_\bR(s)$ (and $\Gamma_\bC(s)=2(2\pi)^{-s}\Gamma(s)$, $\Gamma_\bR(s)=\pi^{-s/2}\Gamma(s/2)$) is the Gamma factor of the motive attached to $\ad f$ (see \cite[Lemma 2.12]{DFG}).
\end{remark}

We recall the definition of Shimura periods attached to Hecke eigenforms. For a subring $A$ of $\bC$, consider the local system $V_k(A)$ on $X$ associated with the $A$-module $\textrm{TSym}^{k-2}(A^{\oplus 2})$. This latter module consists of homogeneous polynomials of degree $k-2$ of the form $P(X,Y)=\sum_{i} \binom{k-2}{i}a_{i}X^iY^{k-2-i}$ with $a_i\in A$ for all $0\leq i \leq k-2$. For such a polynomial $P$, we let $s\in S=M_2(\bZ)\cap \GL_2(\bQ)$ act on $P$ via $(P|s)(X,Y)=P( (X\ Y) (s')^\textrm{tr})$, \textit{i.e.},
	 \[(P|s)(X,Y)=P(dX-bY,-cX+aY) \qquad \mbox{if } s=\begin{pmatrix}a & b \\ c & d\end{pmatrix}.\]
As $\Gamma$ is torsion-free,  $\HH^1(X(\bC),V_k(A))$ can be identified with the parabolic cohomology group  $\HH_P^1(\Gamma,\textrm{TSym}^{k-2}(A^{\oplus 2}))$. It comes equipped with an action of the abstract Hecke algebra $\cH(S,\Gamma)$ (\cite[\S6.3]{hidaLFE}). 

For a given $f\in S_k(M)$, we define classes in $\HH^1(X(\bC),V_k(\bC))$ by the formulae
	\[
	\delta(f)(\gamma)=\int_{z_0}^{\gamma z_0}(X-zY)^{k-2}f(z)\diff\!z, \ \ \delta^c(f)=-\int_{z_0}^{\gamma z_0}(X-\ob{z}Y)^{k-2}\ob{f^c(z)}\diff\! \ob{z},
	\]
where $z_0\in X(\bC)$, $\gamma\in \Gamma$ and $f^c(z)=\ob{f(-\ob{z})}$. The cohomology classes do not depend on the choice of $z_0$. The action of the complex conjugation $[\Gamma\iota\Gamma]$, where $\iota=\left(\begin{smallmatrix}1 & 0 \\ 0 & -1\end{smallmatrix}\right)$, sends $\delta(f)$ to $\delta^c(f)$. The Eichler-Shimura period map is the $\bC$-linear isomorphism 
\begin{equation}\label{eq:eichler_shimura_period_map}
	\begin{array}{ccc}
	S_k(M) \oplus S_k(M) & \overset{\sim}{\longrightarrow} & \HH^1(X(\bC),V_k(\bC)) \\ 
	(f,g) & \mapsto & \delta(f)+\delta^c(g).
\end{array}
\end{equation}
Take $f\in S_k(M)$ and $s=\left(\begin{smallmatrix}a & b \\ c & d\end{smallmatrix}\right)\in S^+$. From 
\[(X-(s\cdot z)Y)^{k-2}|s'=\left(\frac{\det s}{cz+d}\right)^{k-2}(X-zY)^{k-2} \quad \mbox{and}\quad f^c|s=(f|\iota s \iota)^c, \]
one sees that 
\begin{equation}\label{eq:action_slash_cohomology}
	\delta(f)|[\Gamma s \Gamma]=\delta(f|[\Gamma s \Gamma]) \qquad \mbox{and}\qquad \delta^c(f)|[\Gamma s \Gamma]=\delta^c(f|[\Gamma \iota s\iota \Gamma]).
\end{equation}
In particular, \eqref{eq:eichler_shimura_period_map} is equivariant with respect to the actions of $T_n,T^*_n$ and $\langle d \rangle$ for $n,d\geq 1$, $(M,d)=1$. 

For any ring $A\subseteq \bC$, we put 
\[\HH^1_k(M)_A:= \HH^1(X(\bC),V_k(A))\]
and we consider the Hecke algebras 
\[\gh_k(M)_A=A[T_n,\langle \ell \rangle]_{n,\ell\geq 1, (\ell,M)=1}, \qquad \gh^*_k(M)_A=A[T^*_n,\langle \ell \rangle]_{n,\ell\geq 1, (\ell,M)=1},\]
as subspaces of $\End_A(\HH^1_k(M)_A)$.
They are finitely generated commutative $A$-algebras. If $2$ is invertible in $A$, we denote by $\HH^1_k(M)_A^\pm$ the subspaces of $\HH^1_k(M)_A$ where the complex conjugation acts by $\pm1$. For $A=\bC$, the map $f \mapsto \delta^\pm(f)$ identifies $S_k(M)$ with $\HH^1_k(M)_A^\pm$, where we have put 
\begin{equation}\label{eq:definition_delta_pm}
	\delta^\pm(f)=\delta(f)\pm\delta^c(f).
\end{equation}

As $S_k(M)$ admits a basis consisting of modular forms with Fourier coefficients in $\bZ$, for every subring $A\subseteq \bC$ we have $S_k(M)=\bC \otimes_{A} S_k(M)_A$ as $\gh_k(M)_\bC=\bC\otimes_A \gh_k(M)_A$-modules, where 
\[S_k(M)_A=\{f\in S_k(M) \ \colon\ f(q)=\sum_{n\geq 1}a_nq^n\in A\otimes_{\bZ}\bZ[[q]] \}.\] 

Recall that $\ob{\bQ}\subset \ob{\bQ}_p\simeq \bC$.  For a normalized eigenform $f(q)=\sum_{n\geq 1}a_nq^n\in S_k(M,\chi)$, the ring $\bZ_f=\bZ[a_n]_{n\geq 1}$ is a subring of $\ob{\bQ}$. We denote by \[\lambda_f \colon \gh_k(M)_{\bZ_f} \to \bZ_f\]
the system of Hecke eigenvalues of $f$ given by $\lambda_f(T_n)=a_n$ and $\lambda_f(\langle d\rangle)=\chi(d)$ for all $n,d\geq 1$, $(d,M)=1$.

For any field $K\subset \bC$ containing $\bZ_f$ and any $\gh_k(M)_K$-module $M$, let $M[\lambda_f]$ be the $\lambda_f$-eigenspace of $M$. 
Using the maps $\delta^\pm$ in \eqref{eq:definition_delta_pm}, it is clear that $\dim_\bC\HH^1_k(M)^\pm_\bC[\lambda_f]=1$. 

Let $\cO_f$ be the completion of $\bZ_f$ inside $\ob{\bQ}_p\simeq\bC$. The image of $\HH^1_k(M)_{\cO_f}$ inside $\HH^1_k(M)_{\ob{\bQ}_p}$ is a $\cO_f$-lattice, which  we still denote by $\HH^1_k(M)_{\cO_f}$ with a slight abuse of notations. As long as we deal with ordinary eigenforms, this won't cause any problem as the ordinary part of $\HH^1_k(M)_{\cO_f}$ is $\cO_f$-free (see Proposition \ref{prop:hida} below).

\begin{definition}\label{def:canonical_periods}
 Let $f\in S_k(M)$ be a normalized eigenform. Let $\xi_f^\pm$ be a $\cO_f$-basis of 
 \begin{equation}\label{eq:eigenspace_definition_xi_f}
 	\HH^1_k(M)_{\cO_f} \cap \HH^1_k(M)^\pm_{\bC}[\lambda_f].
 \end{equation}
 The $p$-normalized periods of $f$ with respect to $\xi^\pm$ are the complex numbers $\Omega_f^\pm\in\bC^\times$ defined by the relation 
 \[\delta^\pm(f)=\Omega^\pm_f\cdot \xi_f^\pm.\]
\end{definition}

\begin{remark}
	One can actually replace the coefficient ring $\cO_f$ in the above definition with any PID $A\subseteq \bC$ (see e.g. \cite[Rem. 5.4.10]{eigenbook}). When $A$ is a subring of $\ob{\bQ}$, one retrieves the usual definition of complex periods attached to normalized eigenforms.
\end{remark}

We now recall the cohomological interpretation of the Petersson norm of normalized eigenforms. Let $A\subseteq \bC$ be a ring, and consider the nondegenerate $A$-linear pairing $[\cdot,\cdot]$ on $\textrm{TSym}^{k-2}(A^{\oplus 2})$ defined in \cite[(4.2.1)]{ohtacrelle}. It is characterized by the identity 
\[[(uX+vY)^{k-2},(u'X+v'Y)^{k-2}]=(uv'-u'v)^{k-2} \qquad (u,u',v,v'\in A).\]
For $s\in S$ and $P,Q\in \textrm{TSym}^{k-2}(A^{\oplus 2})$, one checks that $[P|s,Q|s]=(\det s)^{k-2}[P,Q]$. Combining $[\cdot,\cdot]$ with the cup-product in cohomology we obtain a new $A$-linear pairing 
\begin{equation}
	(\cdot,\cdot)_\pi \colon \HH^1_k(M)_A \times \HH^1_k(M)_A \to \HH^2(X(\bC),A) \simeq A,
\end{equation}
where the last isomorphism is the cap-product with the fundamental class of $X(\bC)$, endowed with its natural orientation of complex manifold. For $s\in S$, the adjoint of $[\Gamma s \Gamma]$ under $	(\cdot,\cdot)_\pi$ is $[\Gamma s' \Gamma]$, so for the modified anti-symmetric pairing 
\begin{equation}\label{eq:modified_pairing}
	(x,y) \mapsto (x,y|W_M)_\pi,
\end{equation}
the Hecke algebras $\gh_k(M)_A$ and $\gh^*_k(M)_A$ become autoadjoint.

\begin{proposition}\label{prop:calcul_poincare_pairing}
	Let $f,g\in S_k(M)$. We have 
	\begin{align*}
		(\delta(f),\delta^c(g))_\pi&= (2i)^{k-1}\langle f,g^c\rangle_\Gamma \\
		(\delta^+(f),\delta^-(g)|W_M)_\pi&=-2(2i)^{k-1}\langle f,g^c|W_M\rangle_\Gamma.
	\end{align*}
\end{proposition}
\begin{proof}
	We start with the first equality. Over $\bC$, the cap-product with the fundamental class of $X(\bC)$ is given by the integration over $X(\bC)$ of cohomology classes seen as $2$-differential forms. Therefore,
	\begin{align*}
		(\delta(f),\delta^c(g))_\pi &= - \int_{X(\bC)}   f(z)\ob{g^c(z)} [(X-zY)^{k-2},(X-\ob{z}Y)^{k-2}] dz\wedge d\ob{z} \\
		&= 2i \int_{X(\bC)} f(z)\ob{g^c(z)}(z-\ob{z})^{k-2} dxdy \\
		&=(2i)^{k-1}\int_{X(\bC)} f(z)\ob{g^c(z)}y^{k-2} dxdy.
	\end{align*}
This proves the first formula. A similar computation as above yields \small
\begin{align*}
	(\delta(h),\delta(h'))_\pi&=0=(\delta^c(h),\delta^c(h'))_\pi, \\  (\delta(h),\delta^c(h'))_\pi&=(-1)^{k-1}(\delta^c(h'),\delta(h))_\pi,
\end{align*}
\normalsize for any $h,h'\in S_k(M)$, and from \eqref{eq:action_slash_cohomology} we find $\delta^-(g)|W_M=\delta(g|W_M)+(-1)^{k-1}\delta^c(g|W_M)$. Hence, by linearity,
\begin{align*}
	(\delta^+(f),\delta^-(g)|W_M)_\pi&= (-1)^{k-1} \left( (\delta(f),\delta^c(g|W_M))_\pi + (\delta(g|W_M),\delta^c(f))_\pi \right) \\
									&=(-2i)^{k-1}\left(\langle f,(g|W_M)^c\rangle_\Gamma+\langle g|W_M,f^c\rangle_\Gamma\right) \\
									&= 2(-2i)^{k-1}\langle f,(g|W_M)^c\rangle_\Gamma \\
									&=-2(2i)^{k-1}\langle f,g^c|W_M\rangle_\Gamma,
\end{align*}
since $(g|W_M)^c=(-1)^kg^c|W_M$. This proves the second formula.
\end{proof}

\subsection{Hida families}

Let $p$ be an odd prime and let $M$ be of the form $Np^r$ with $(N,p)=1$ and $r\geq 1$. We let $X_r=X(\Gamma_1(Np^r))_{/\bC}$. We shall consider the ordinary and anti-ordinary parts of the (co-)homology of $X_r$, defined with the help of Hida's projectors
\[ e=\lim_r (T_p)^{r!}, \qquad e^*=\lim_r (T^*_p)^{r!}.\]
In what follows, we drop for convenience the subscript $\Zp$ when we work with cohomology with $\Zp$-coefficients.
\begin{proposition}[Hida]\label{prop:hida}
	The Atkin-Lehner map 
	\[\cdot |W_{Np^r} \colon e\HH^1_k(Np^r)\to e^*\HH^1_k(Np^r)\]
	is a $\Zp$-linear isomorphism, and $e\HH^1_k(Np^r)$ and $e^*\HH^1_k(Np^r)$ are $\Zp$-free.
\end{proposition}
\begin{proof}
	See \cite[Corollary 3.2]{hidamodules} and \cite[(2.4) p.347]{hidamodules}.
\end{proof}
In particular, we have $e\gh(Np^r)=e^*\gh^*(Np^r)$ via $T_n \leftrightarrow T^*_n$.
Poincaré duality yields a canonical identification 
\[ \HH_1(X_r(\bC),\Zp) = \HH^1 (X_r(\bC),\Zp), \]
which is compatible with the Hecke action, in the sense that the action of $T_n$ and $\langle d \rangle$ on the left hand side corresponds to that of $T_n^*$ and $\langle d \rangle^{-1}$ on the right hand side respectively. By functoriality of (co)homology, we obtain 
\[ \varprojlim_r e \HH_1(X_r(\bC),\Zp) =\varprojlim_r e^* \HH^1 (X_r(\bC),\Zp) =: e^* \HH^1 (Np^\infty), \]
as Hecke-modules, where the transition maps on cohomology are the trace maps. Since the adjoint Hecke operators are compatible with trace maps in cohomology, $e^* \HH^1 (Np^\infty)$ inherits a structure of module over 
\[e^*\gh^*(Np^\infty)=\varprojlim_r e^*\gh_2^*(Np^r).\]
We will primarily be interested in $e^*\gh^*(Np^\infty)$, but note that it is canonically isomorphic to $e\gh(Np^\infty)=\varprojlim_r e\gh_2(Np^r)$. 

The action of inverse diamond operators endows the ring $e^*\gh(Np^\infty)$ with a structure of algebra over $\tilde{\Lambda}=\varprojlim_r\Zp[(\bZ/Np^r)^\times]$. Letting $U_r=1+p^r\Zp$ for all $r\geq 1$, we have $\tilde{\Lambda}=\Lambda[(\bZ/Np)^\times]$ where we put 
\[\Lambda=\Zp[[U_1]]=\varprojlim_r \Zp[U_1/U_r],\]
and $e^*\gh(Np^\infty)$ is known to be finite free over $\Lambda$ by \cite[Thm. 3.1]{hida1986galois}.

Let $\bI$ be a finite flat extension of $\Lambda$ and let $\gamma$ be a topological generator of $U_1$. We identify via $\bC\simeq \ob{\bQ}_p$ Dirichlet characters of $p$-power order and conductor with elements of $\widehat{U_1}:=\Hom_{\textrm{cts}}(U_1,\ob{\bQ}_p^\times)=\Hom_{\Zp-\textrm{alg}}(\Lambda,\ob{\bQ}_p)$. The set of arithmetic points of $\bI$ is defined as 
\[\cX(\bI)=\left\{ \nu\in\Hom_{\Lambda}(\bI,\ob{\bQ}_p) \ | \ \exists k\geq 2, \ \exists \varepsilon\in \widehat{U_1},\ \nu(\gamma)=\gamma^{k-2}\epsilon(\gamma) \right\}.\]
We refer to the pair $(k,\epsilon)$ in the definition of $\nu\in\cX(\bI)$ as the weight of $\nu$, and we define a finite flat extension of $\Zp$ by letting $\cO_\nu=\nu(\bI)$. We further denote by $S(N,\bI)$ the space of ordinary $\bI$-adic cusp forms of tame level $N$. Recall that $S(N,\bI)$ is spanned over $\bI$ by formal $q$-expansions $F(q)=\sum_{n\geq 1} a_nq^n\in\bI[[q]]$ with the property that there exists a Dirichlet character $\theta$ of level dividing $Np$ such that, for all $\nu\in\cX(\bI)$ of weight $(k,\varepsilon)$, the specialization of $F$ at $\nu$ yields
\begin{equation}\label{eq:specialisation_Hida_family}
	F_\nu:=\sum_{n\geq 1} \nu(a_n)q^n\in e S_k(Np^{r_\nu},\theta\varepsilon\omega^{2-k}). 
\end{equation}
Here, $\omega$ is the Teichmüller character, $p^{r_\nu}=\max(p,C(\varepsilon))$ (and $C(\varepsilon)$ is the conductor of $\varepsilon$), and $\cO_\nu$ is seen inside $\bC$. 

On $S(N,\bI)$ there is an action of Hecke operators $T_n$ which is compatible with specialization maps. More precisely, $S(N,\bI)=S(N,\Lambda)\otimes \bI$ and the $\Lambda$-subalgebra of $\End_\Lambda(S(N,\Lambda))$ generated by the $T_n$'s is canonically isomorphic to $e\gh(Np^\infty)_{\Zp}$. Furthermore, for all $k\geq 2$ and $r\geq 1$, 
\begin{align*}
	e\gh(Np^\infty) \otimes_{\Lambda} \Lambda/(\omega_{k,r}) &\simeq e\gh_k(Np^r),
\end{align*}
by Hida's control theorem \cite[Thm. 1.2]{hida1986galois}, where $\omega_{k,r}\in\Lambda$ is defined by $\omega_{k,r}(\gamma)=[\gamma]^{p^{r-1}}-\gamma^{(k-2)p^{r-1}}$. In particular, $e\gh(Np^\infty) \simeq \varprojlim_r e\gh_k(Np^r)$ for all integers $k\geq 2$.

A primitive Hida family is a $q$-expansion $F$ as above which is a normalized eigenform (\textit{i.e.}, $a_1=1$, $T_nF=a_nF$ for all $n\geq 1$) and such that $F_{\nu_0}$ is a primitive newform for at least one $\nu_0\in \cX(\bI)$. The character $\theta$ in \eqref{eq:specialisation_Hida_family}, called the tame character of $F$, is necessarily even, and we have $a_p\in \bI^\times$. Moreover, given $\nu\in\cX(\bI)$ of weight $(k,\varepsilon)$, if $F_\nu^\circ$ denotes the newform associated with $F_\nu$, then the following is known. If $p$ divides $C(\theta\varepsilon\omega^{2-k})$, then $F_\nu=F_\nu^\circ$. Otherwise, $\varepsilon=1$, $\theta\omega^{2-k}$ is of prime-to-$p$ conductor and either $F_\nu=F_\nu^\circ$ and $k=2$, or 
\[F_\nu(q)=F_\nu^\circ(q)-\beta_\nu F_\nu^\circ(q^p),\]
where $\beta_\nu$ is the non-unit root of the polynomial $X^2-\nu(a_p)X+\theta\omega^{2-k}(p)p^{k-1}$ (\cite[Thm. 4.1]{hidaAIF}).

Every normalized eigenform $F(q)=\sum_{n\geq 1}a_nq^n \in S(N,\bI)$ gives rise to a homomorphism $\lambda_F\colon e\gh(Np^\infty) \to \bI$ sending $T_n$ to $a_n$. The kernel $\ga$ of $\lambda_F$ is a minimal prime, and is contained in a unique maximal prime $\gm$ of $e\gh(Np^\infty)$. One may assume that the ring of coefficients of $F$ is an integrally closed local domain by taking $\bI$ to be the normalization of $\bT/\ga$, where 
\[\bT:=e\gh(Np^\infty)_\gm.\]

Let $\bT'$ be the subalgebra of $\bT$ generated over $\Lambda$ by $T_n$ for all $n\geq1$ coprime with $Np$. Then $\bT'$ is reduced and it is finitely generated and torsion-free over $\Lambda$. Let 
\[\rho_\gm \colon G_{\bQ} \to \GL_2(\bT'\otimes_{\Lambda}\Frac(\Lambda))\]
be Hida's big ordinary representation and let $\ob{\rho}_\gm$ be (the semisimplification of) its residual representation modulo $\gm$. When $\ob{\rho}_\gm$ is absolutely irreducible, it is known by \cite[Théorème 1]{nyssen1996pseudo} that $\rho_\gm$ is conjugate to a $\bT'$-valued representation. 

In what follows it is more convenient to work with the anti-ordinary big Hecke algebra $e^*\gh^*(Np^\infty)$. Since $e^*\gh^*(Np^\infty)=e\gh(Np^\infty)$ canonically, $F$ gives rise to a triple $(\lambda_F^*,\ga^*,\gm^*)$ defined in the obvious way. 
\begin{proposition}\label{prop:EPW}
	Let $\bT^*=e^*\gh^*(Np^\infty)_{\gm^*}$ be a local component of the big Hecke algebra such that $\ob{\rho}_\gm$ satisfies \eqref{tag:cr}. Then $e^*\HH^1(Np^\infty)^\pm_{\gm^*}$ is free of rank $1$ over $\bT^*$.
\end{proposition}
\begin{proof}
	This is a reformulation in terms of cohomology groups of \cite[Prop. 3.3.1]{emerton2006variation}.
\end{proof}

\subsection{Interpolation of Petersson norms}\label{sec:interpolation_petersson}
We keep the notations of the preceding sections. For any $k\geq 2$ and $r\geq 1$, we see $e^*\HH_k^1(Np^r)$ as a $\Lambda$-module by letting $d\in U_1$ act  by $ d^{k-2}\langle d \rangle^{-1}$. 
Then $\omega_{k,r}$ acts trivially on it, and \cite[Thm. 1.4.3 (ii)]{ohtacrelle} yields a $\Lambda/\omega_{k,r}\Lambda$-linear isomorphism
\begin{equation}\label{eq:control_theorem_cohomology}
	e^*\HH^1(Np^\infty) \otimes_\Lambda \Lambda/(\omega_{k,r}) \simeq e^*\HH_k^1(Np^r).
\end{equation}
In \cite[Thm. 4.2.5]{ohtacrelle}, Ohta constructed a $\Lambda$-linear pairing on $e^*\HH^1(Np^\infty)$ which interpolates the Poincaré pairings at level $Np^r$ for all $r\geq 1$. 

Recall that $e^*\gh^*(Np^\infty)$ is a complete semi-local ring, hence it is the finite product of its localizations at maximal ideals. It is enough for our purpose to fix one of its local components $\bT^*=e^*\gh^*(Np^\infty)_{\gm^*}$ which we see as a $\Lambda$-algebra. 
\begin{theorem}[Ohta]
	There exists a pairing 
	\begin{equation}\label{eq:ohta_pairing}
		(\cdot,\cdot)_{\Lambda} \colon e^*\HH^1(Np^\infty)_{\gm^*} \times e^*\HH^1(Np^\infty)_{\gm^*} \to \Lambda
	\end{equation}
	fitting, for all $k\geq 2$ and $r\geq 1$, in the commutative diagram
	\[\begin{tikzcd}
		e^*\HH^1(Np^\infty)_{\gm^*} \times e^*\HH^1(Np^\infty)_{\gm^*} \rar\dar & \Lambda \dar \\ 
		 e^*\HH_k^1(Np^r) \times e^*\HH_k^1(Np^r) \rar & \Zp[U_1/U_r],
	\end{tikzcd}\]
	where the bottom horizontal map is given by
	\begin{equation}\label{eq:ota_pairing_interpolation}
			(x,y)  \mapsto (-1)^{k}\cdot \sum_{d\in U_1/U_r} (dNp^r)^{2-k}\cdot (x,y|(T_p^*)^r|W_{Np^r}|\langle d \rangle^{-1})_\pi \cdot [d],
	\end{equation}
	the left vertical map is induced by \eqref{eq:control_theorem_cohomology} and the right vertical map sends $d\in U_1$ to $d^{k-2}\cdot[d]$. Moreover, $(\cdot,\cdot)_{\Lambda}$ is a $\Lambda$-bilinear perfect pairing satisfying $(x|T_n^*,y)_{\Lambda}=(x,y|T_n^*)_{\Lambda}$ for all integers $n\geq 1$.

\end{theorem}
\begin{proof}
	See \cite[Thm. 4.2.5 and Cor. 4.2.8]{ohtacrelle}. The factor $(Np^r)^{2-k}$ comes from the difference of normalizations in the definition of $W_{Np^r}$ (see \cite[(4.1.3)]{ohtacrelle}). The perfectness of Ohta's pairing follows from that of Poincaré's pairing on $\HH^1_2(Np)=\HH^1(X_1(Np)(\bC),\Zp)$ and Hida's control theorem. 
\end{proof}

Choose a finite flat extension $\bI\subset \ob{\Frac(\Lambda)}$ of $\Lambda$ containing the Fourier coefficients of all the normalized eigenforms in $S(N,\ob{\Frac(\Lambda)})$ with residual Galois representation isomorphic to $\ob{\rho}_\gm$. Replacing $\bI$ with its integral closure in $\ob{\Frac(\Lambda)}$ if necessary, we assume that $\bI$ is normal. Let 
\[\bT_\bI^*=e^*\gh^*(Np^\infty)_{\gm^*}\otimes_{\Lambda}\bI, \qquad \cM^\pm:=e^*\HH^1(Np^\infty)^\pm_{\gm^*}\otimes_{\Lambda}\bI.\]
Given a normalized eigenform $F(q)=\sum_{n\geq 1}a_nq^n\in S(N,\bI)$ with residual Galois representation $\ob{\rho}_\gm$, the extension of scalars of $\lambda_F^*$ to $\bI$ induces a $\bI$-linear surjection $\bT^*_\bI \twoheadrightarrow \bI$ which we still denote by $\lambda_F^*$. Let $\cK$ be the fraction field of $\bI$ and let $\bT_\cK^*=\bT_\bI^* \otimes_{\bI} \cK$. We make the following assumption on $F$: there exists an $\cK$-algebra decomposition
\begin{equation}
	\bT_\cK^* \simeq \cK \times X, \tag{pr}\label{tag:pr}
\end{equation}
where the first projection, say $e_F$, is induced by $\lambda_F^* \otimes \cK$. We further assume that
\begin{equation}
	\cM^\pm \text{ is $\bT_\bI^*$-free of rank 1}, \tag{fr}\label{tag:fr}
\end{equation}
and, following the notations of Appendix \ref{appendix}, we let $\cM^\pm_{\lambda_F^*}:=\cM^\pm \cap e_F(\cM^\pm \otimes_{\bI}\cK)$. In particular, $\cM^\pm_{\lambda_F^*}$ is $\bI$-free of rank $1$ by Lemma \ref{lem:appendix_direct_summand}.

\begin{remark}\label{rem:conditions_pr_and_fr}
	Condition \eqref{tag:pr} holds if $F$ is primitive of level $N$ by \cite[Thm. 4.2]{hidaAIF}. It also holds if $\bT_\cK^*$ is semi-simple (\textit{i.e.}, if $\bT^*$ is reduced), which happens for instance when $N$ is cube-free \cite[Cor. 1.3]{hida2013inventiones}. Condition \eqref{tag:fr} holds if $\ob{\rho}_\gm$ satisfies \eqref{tag:cr} by Proposition \ref{prop:EPW}.
\end{remark}
In the following definition, we write $(\cdot,\cdot)_\bI \colon \cM^+ \times \cM^- \to \bI$ for the $\bI$-linear extension of Ohta's pairing \eqref{eq:ohta_pairing}.

\begin{definition}\label{def:p_adic_adjoint_L_function}
	Assume \eqref{tag:pr} and \eqref{tag:fr} and let $\xi_F^\pm$ be a $\bI$-basis of $\cM^\pm_{\lambda_F^*}$. We define the $p$-adic adjoint $L$-function of $F$ by letting 
	\[L_p(\ad F)=(\xi^+_F,\xi^-_F)_\bI\in\bI.\]
For any $\nu \in\Hom(\bI,\ob{\bQ}_p)$, we shall write $L_p(\ad F,\nu):=\nu(L_p(\ad F))$.
\end{definition}

\begin{proposition}\label{prop:congruence_module_and_p_adic_L_functions}
		Assume \eqref{tag:pr} and \eqref{tag:fr}. Then the congruence module of $F$ is isomorphic to $\bI/L_p(\ad F)\bI$. 
\end{proposition}
\begin{proof}
	The congruence module of $F$ is, by definition, the module $C_0(\lambda_F)$ over $\bT_\bI=e\gh(Np^\infty)_{\gm}\otimes_\Lambda \bI$ defined in Appendix \ref{appendix}, where $\lambda_F \colon \bT_\bI \twoheadrightarrow \bI$ is the $\bI$-algebra homomorphism sending $T_n$ to $a_n$. Under the identification $\bT=\bT^*$, we then have $C_0(\lambda_F)=C_0(\lambda_F^*)$, which, in turn, is isomorphic to $\bI/L_p(\ad F)\bI$ by Corollary \ref{coro:appendix_generator_congruence_module}.
\end{proof}
Given a primitive Hida family $F$, we now express the special values of $L_p(\ad F)$ at arithmetic weights in terms of adjoint $L$-values. We still assume \eqref{tag:fr}. 

We first define $p$-normalized periods as follows. Let $\nu\in\cX(\bI)$ be of weight $(k,\varepsilon)$, let $r=r_\nu$ and fix a sign $\pm$. Let $\bT^*_\nu$ be the local factor of $e^*\gh_k(Np^r)_{\cO_\nu}$ through which $\lambda^*_{F_\nu}$ factors, and let $\cM^\pm_\nu$ be the corresponding direct summand of $e^*\HH^1_k(Np^r)_{\cO_\nu}$. 
Let also $a$ be the sign of $(-1)^k$ and $\pm a$ be the sign of $\pm (-1)^k$.
We have
\begin{equation}\label{eq:isom_specialization_xi}
	\cM^\pm_{\lambda^*_F} \otimes_{\bI,\nu} \cO_\nu \simeq (\cM^\pm_\nu)_{\lambda^*_{F_\nu}} \simeq (e\HH^1_k(Np^r)^{\pm a}_{\cO_\nu})_{\lambda_{F_\nu}},
\end{equation}
where the first map is the specialization map and the second one is the map $W_{Np^r}^{-1}$. The first isomorphism is Hida's control theorem \cite[Thm. 4.2]{hidaAIF}, while the second follows from Proposition \ref{prop:hida}. Note that the last module in \eqref{eq:isom_specialization_xi} coincides with that in \eqref{eq:eigenspace_definition_xi_f} with $M=Np^r$ and $f=F_\nu$.
\begin{definition}\label{def:family_canonical_periods}
	For each $\nu\in\cX(\bI)$, we define $\xi_{F_\nu}^\pm$ as the image of $\xi^\pm_F$ under the map in \eqref{eq:isom_specialization_xi}, and we let $\Omega^\pm_{F_\nu}\in\bC^\times$ be the associated $p$-normalized canonical period of Definition \ref{def:canonical_periods}. In other words, $\xi_{F_\nu}^{\pm}|W_{Np^{r_\nu}}$ is the specialization of $\xi_{F}^{\pm a}$ at $\nu$ (where $a$ is the sign of $(-1)^k$), and $\delta^\pm(F_\nu)=\Omega^\pm_{F_\nu}\cdot\xi_{F_\nu}^\pm$.
\end{definition}
\begin{remark}\label{rem:periods_associated_newforms}
	For a weight $\nu\in\cX(\bI)$ such that $F_\nu\in S_k(Np)$ is $p$-old, the $p$-normalized canonical period of $F_\nu^\circ$ can be defined by simply letting $\Omega^\pm_{F_\nu^\circ}=\Omega^\pm_{F_\nu}$. Indeed, by \cite[Prop. 7.3.1]{KLZcambridge} there is a ``$p$-stabilization isomorphism'' $(-)^\circ\colon\HH^1_k(Np)_\bC\simeq \HH^1_k(N)_\bC$ sending $\delta(F_\nu)$ to $\delta(F_\nu^\circ)$ and identifying the lattices $(\HH^1_k(Np)_{\cO_\nu})_{\lambda_{F_\nu}}$ and $(\HH^1_k(N)_{\cO_\nu})_{\lambda_{F_\nu^\circ}}$. Therefore, the periods of $F_\nu$ and of $F_\nu^\circ$ with respect to $\xi_{F_\nu}^\pm$ and $(\xi_{F_\nu}^\pm)^\circ$ respectively are equal.
\end{remark}
\begin{theorem}\label{thm:calcul_valeurs_speciales}
	Assume \eqref{tag:fr} holds and suppose $F$ is primitive as before. Let $\nu\in\cX(\bI)$ be of weight $(k,\varepsilon)$, let $r=r_\nu$ and write the conductor of $F_\nu^\circ$ as $Np^{r_0}$. Put $\alpha_\nu=a_p(F_\nu)$, $\beta_\nu=\frac{(\theta\omega^{2-k}\varepsilon)(p)p^{k-1}}{\alpha_\nu}$ and
	\[\cE_p(\ad F_\nu^\circ)=\alpha_\nu(p-1)(1-\frac{\beta_\nu}{\alpha_\nu})(1-p^{-1}\frac{\beta_\nu}{\alpha_\nu}) \]
	if $F_\nu\neq F_\nu^\circ$ and $\cE_p(\ad F_\nu^\circ)=1$ otherwise. We have
	\begin{equation}
		L_p(\ad F,\nu)=p^{r-1}\alpha_\nu^r w_\nu\cE_p(\ad F_\nu^\circ)\cdot i(-2i)^{k}\dfrac{ \langle F_\nu^\circ,F_\nu^\circ \rangle_{\Gamma_1(Np^{r_0})}}{\Omega^+_{F^\circ_\nu}\Omega^-_{F^\circ_\nu}},
	\end{equation}
	where $w_\nu\in \ob{\bQ}^\times$ is given by $F_\nu^\circ|W_{Np^{r_0}}=w_\nu (F_\nu^\circ)^c$ and $\Omega_{F^\circ_\nu}^\pm$ is as in Definition \ref{def:family_canonical_periods} and Remark \ref{rem:periods_associated_newforms}.
\end{theorem}
\begin{proof}
	Let $W=W_{Np^r}$ for simplicity. Let $a$ be the sign of $(-1)^k$. Since $\nu$ factors through $\bI/\omega_{k,r}\bI$, $L_p(\ad F,\nu)$ equals
\[(-1)^{k}\cdot \sum_{d\in U_1/U_r} (dNp^r)^{2-k}\cdot (\xi_{F_\nu}^{+a}|W,\xi_{F_\nu}^{-a}|W|(T_p^*)^r|W|\langle d \rangle^{-1})_\pi \cdot d^{k-2}\varepsilon(d).\]
Noticing that $\xi_{F_\nu}^{-a}|T_p=\alpha_\nu\xi_{F_\nu}^{-a}$ and $\xi_{F_\nu}^{-a}|\langle d\rangle^{-1}=\varepsilon^{-1}(d)\xi_{F_\nu}^{-a}$, we obtain 
\begin{align*}
	L_p(\ad F,\nu)&= (-1)^{k}(Np^r)^{2-k} p^{r-1}\alpha_\nu^r\cdot (\xi_{F_\nu}^{+a}|W,\xi_{F_\nu}^{-a}|W^2)_\pi \\
					&= (-1)^{k} p^{r-1}\alpha_\nu^r\cdot (\xi_{F_\nu}^{+a},\xi_{F_\nu}^{-a}|W)_\pi \\
					&= p^{r-1}\alpha_\nu^r\cdot (\xi_{F_\nu}^{+},\xi_{F_\nu}^{-}|W)_\pi,
\end{align*}
since the adjoint of $W$ for $(\cdot,\cdot)_\pi$ is $Np^r\cdot W^{-1}$ and \eqref{eq:modified_pairing} is anti-symmetric. Hence, 
\[	L_p(\ad F,\nu)= -2(2i)^{k-1} p^{r-1}\alpha_\nu^r\cdot \dfrac{\langle F_\nu,F^c_\nu|W\rangle_{\Gamma_1(Np^r)}}{\Omega^+_{F^\circ_\nu}\Omega^-_{F^\circ_\nu}},\]
by Proposition \ref{prop:calcul_poincare_pairing}. We are reduced to compute $\frac{\langle F_\nu,F^c_\nu|W\rangle_{\Gamma_1(Np^r)}}{\langle F^\circ_\nu,F^\circ_\nu\rangle_{\Gamma_1(Np^{r_0})}}=c\cdot\frac{\langle F_\nu,F^c_\nu|W\rangle_{\Gamma_0(Np^r)}}{\langle F^\circ_\nu,F^\circ_\nu\rangle_{\Gamma_0(Np^{r_0})}}$, where $c=1$ if $F_\nu=F_\nu^\circ$ and $c=p-1$ otherwise. This is already done in \cite[(9.5) p.79]{hidaAIF}, which yields:
\begin{equation*}
	\frac{\langle F_\nu,F^c_\nu|W\rangle_{\Gamma_0(Np^r)}}{\langle F^\circ_\nu,F^\circ_\nu\rangle_{\Gamma_0(Np^{r_0})}}=
	\left\{\begin{array}{ll}
		(-1)^kw_\nu  &\textrm{if $F_\nu=F_\nu^\circ$}\\
	(-1)^kw_\nu \alpha_\nu(1-\frac{\beta_\nu}{\alpha_\nu})(1-p^{-1}\frac{\beta_\nu}{\alpha_\nu})  &\textrm{if $F_\nu\neq F_\nu^\circ$}.
	\end{array}\right.
\end{equation*}
\end{proof}

\subsection{A $p$-adic $L$-function for varying branches}
We keep the notations of the preceding sections and we fix in particular a local component $\bT$ of $e\gh(Np^\infty)$ with residual Galois representation $\ob{\rho}=\ob{\rho}_\gm$. We let as before $\bT^*$ denote the corresponding local component of $e^*\gh^*(Np^\infty)$ under $e\gh(Np^\infty)\simeq e^*\gh^*(Np^\infty)$, and $\bT^*_\bI=\bT^*\otimes_{\Lambda}\bI$. Put also  $\cM=e^*\HH^1(Np^\infty)_{\gm^*}\otimes_{\Lambda}\bI$.

Assuming \eqref{tag:fr}, we discuss in this paragraph the existence of an element $L_p(\ad \ob{\rho})\in\bT_\bI$ such that, for every normalized eigenform $F$ of level $N$ with residual representation $\ob{\rho}$ and satisfying \eqref{tag:pr}, we have $\lambda_F(L_p(\ad \ob{\rho}))=L_p(\ad F)$. We shall in particular explain how to make a consistent choice of canonical periods for every such eigenform $F$.

Following Fukaya and Kato \cite[\S1.6]{fukayakato}, we define a $\bI$-linear pairing on $\cM \times \cM$ by letting
\[((x,y))_\bI=\sum_{n\geq 1} (x,y|T_n^*)_\bI \cdot q^n\in S(N,\bI)_\gm\]
for all $(x,y)\in\cM \times \cM$.
\begin{proposition}\label{prop:pairing_fukaya_kato}
	The pairing $((\cdot,\cdot))_\bI \colon \cM \times \cM \to S(N,\bI)_\gm$ is perfect and it satisfies
	\[((x|T_m^*,y))_\bI=((x,y|T_m^*))_\bI=((x,y))_\bI|T_m\]
	for all $m\geq1$ and $(x,y)\in\cM \times \cM$. 
\end{proposition}
\begin{proof}
	See \cite[Prop. 1.6.6]{fukayakato}.
\end{proof}
Hida duality for $\Lambda$-adic cusp forms \cite[Thm. 1.3]{hidaAIF} implies $S(N,\Lambda)_\gm\simeq \Hom_{\Lambda}(\bT,\Lambda)$. Hence, under the assumption that $\bT$ is Gorenstein, we see that $S(N,\bI)_\gm$ is $\bT_\bI$-free of rank $1$. This Gorensteinness condition is implied by \eqref{tag:fr}. Indeed, Ohta's pairing provides an isomorphism $\cM^+\simeq\Hom_\bI(\cM^-,\bI)$, hence $\bT_\bI\simeq\Hom_\bI(\bT_\bI,\bI)$. Therefore, the perfectness of $((\cdot,\cdot))_\bI$ shows that  $((\xi^+_\gm,\xi^-_\gm))_\bI$ generates $S(N,\bI)_\gm$ as a $\bT_\bI$-module if $\xi_\gm^\pm$ generates $\cM^\pm$ as a $\bT_\bI^*$-module.
\begin{definition}\label{def:formula_adjoint_rho_bar}
	Assume \eqref{tag:fr} and let $\xi_\gm^\pm$ be a generator of $\cM^\pm$ as a $\bT_\bI^*$-module. Put $G_\gm:=((\xi^+_\gm,\xi^-_\gm))_\bI$. We define the adjoint $p$-adic $L$-function of $\ob{\rho}$ as the unique element $L_p(\ad \ob{\rho})\in\bT_\bI$ such that
	\begin{equation}\label{eq:formula_adjoint_rho_bar}
		\sum F=L_p(\ad \ob{\rho})\cdot G_\gm,
	\end{equation}
	where the sum runs over the set of all normalized eigenforms with ordinary residual representation isomorphic to $\ob{\rho}$. 
\end{definition}
\begin{theorem}\label{thm:specialization_Hida_eigenforms}
	Assume \eqref{tag:fr} and let $\xi_\gm^\pm$ be a generator of $\cM^\pm$ as a $\bT_\bI^*$-module. For any normalized eigenform $F$ satisfying \eqref{tag:pr}, we let $\xi_F^\pm= e_F\cdot (L_p(\ad \ob{\rho})\xi^\pm_\gm)$, where $e_F$ is the projector associated with $F$. Then $\xi_F^\pm$ is a generator of $\cM^\pm_{\lambda^*_F}$, and 
	\[\lambda_F(L_p(\ad \ob{\rho}))=L_p(\ad F), \]
	where $L_p(\ad F)$ is computed with respect to $\xi^\pm_F$.
\end{theorem}
\begin{proof}
	Assume $F(q)=\sum_{n\geq 1}a_nq^n$ satisfies \eqref{tag:pr}. Let $e_F,e'_F\in\bT_\cK=\bT_\bI\otimes_\bI\cK$ be the projectors of the algebra decomposition $\bT_\cK\simeq\cK\times X$, where $\cK=\Frac(\bI)$. We put $L_F=\lambda_F(L_p(\ad \ob{\rho}))\in\bI$ for simplicity. Hence, $\xi_F^\pm= L_F e_F\cdot \xi^\pm_\gm$. By multiplicity one, $e_FS(N,\cK)_\gm=\cK.F$,  so
		\begin{equation}\label{eq:relation_F_G_m}
			F= e_F\cdot L_p(\ad \ob{\rho})G_\gm=L_F e_F\cdot G_\gm. 
		\end{equation}
		Since $S(N,\bI)_\gm$ is $\bT_\bI$-free and generated by $G_\gm$, we thus have 
		\[C_0(\lambda_F)\simeq (S(N,\bI)_\gm)^{\lambda_F}/(S(N,\bI)_\gm)_{\lambda_F}=(\bI.e_F\cdot G_\gm)/(\bI.F)=\bI/(L_F). \]
		As $C_0(\lambda_F)\simeq (\cM^\pm)^{\lambda_F^*}/\cM^\pm_{\lambda_F^*}$, we deduce that $\xi^\pm_F$ generates $\cM^\pm_{\lambda_F^*}$ since $e_F\cdot\xi^\pm_\gm$ generates $(\cM^\pm)^{\lambda_F^*}$.
		
		We now prove that $L_p(\ad F):=(\xi^+_F,\xi^-_F)_\bI$ equals $L_F$. Denote by a subscript $\cK$ the extension of scalars to $\cK$ of the pairings $(\cdot,\cdot)_\bI$ and $((\cdot,\cdot))_\bI$. Using $e_F^2=e_F$ and Proposition \ref{prop:pairing_fukaya_kato}, we see that 
		\begin{align*}
			e_F\cdot G_\gm &= ((e_F\cdot\xi_\gm^+,e_F\cdot\xi_\gm^-))_\cK \\
			&=\sum_{n\geq 1}(e_F\cdot \xi_\gm^+,e_F\cdot \xi_\gm^-|T_n^*)_\cK\cdot q^n \\
			&=\sum_{n\geq 1}(e_F\cdot \xi_\gm^+,e_F\cdot \xi_\gm^-)_\cK\cdot a_n q^n=(e_F\cdot \xi_\gm^+,e_F\cdot \xi_\gm^-)_\cK\cdot F.
		\end{align*}
		Hence, multiplying the last equality by $L_F^2$ and combining it with \eqref{eq:relation_F_G_m}, one obtains
		\[L_F\cdot F=(\xi^+_F,\xi^-_F)_\cK\cdot F = L_p(\ad F)\cdot F,\]
		so $L_F= L_p(\ad F)$ as claimed.
\end{proof}

\section{The main conjecture}\label{sec:section_3_main_conjecture}
\subsection{Twists and change of tame level}
Fix a large enough integrally closed finite flat extension $\bI$ of $\Lambda$ as in \S\ref{sec:interpolation_petersson}. Let $F\in S(N,\bI)$ be a primitive Hida family of tame character $\theta$ whose residual Galois representation satisfies \eqref{tag:cr}. Then there exists a continuous Galois representation $\rho_F \colon G_{\bQ} \to \GL_2(\bI)$ unramified outside $Np$ such that 
\[\forall \ell 	\nmid Np,\quad \left\{\begin{array}{lcl}\tr\rho_F(\sigma_\ell)&=&a_\ell(F) \\
	\det\rho_F(\sigma_\ell)&=&\ell\cdot \theta_\bI(\ell),
\end{array}\right.\]
where $\sigma_\ell$ is the arithmetic Frobenius at $\ell$ and $\theta_\bI \colon (\bZ/Np)^\times \to \bI^\times$ is such that $\nu(\theta_\bI(\ell))=(\theta\omega^{2-k}\varepsilon)(\ell)\ell^{k-2}$ for any $\nu\in\cX(\bI)$ of weight $(k,\varepsilon)$. In particular, the specialization $\rho_{\nu}$ of $\rho_F$ at any $\nu\in\cX(\bI)$ is isomorphic to the usual Eichler-Shimura-Deligne representation attached to $F_\nu$.

Given a character $\psi$ of prime-to-$p$ conductor, there exists a unique primitive Hida family $F_\psi$ underlying $F \otimes \psi$. In particular, $F_\psi$ has tame character $\theta\psi^2$ and if $\psi$ is unramified at a prime $\ell\nmid N$, then $a_\ell(F_\psi)=\psi(\ell)a_\ell(F)$. We say that $F$ is twist-minimal if $F$ has minimal level among its twists. 

Let $\pi=\otimes'_\ell \pi_\ell$ be the automorphic representation of $\GL_2(\bA)$ generated by $F^\circ_\nu$ for some arithmetic weight $\nu\in\cX(\bI)$. By \cite[\S6.1]{dimitrov2014local}, the automorphic type of $\pi_\ell$ at a certain prime $\ell$ does not depend on the choice of $\nu$. For any prime $\ell$, we then define
\small 
\begin{equation}\label{eq:euler_factor}
	E_\ell(\ad F)=\left\{\begin{array}{ll}
	 (\ell-1)(a_\ell(F)^2-\theta_\bI(\ell)(1+\ell)^2)& \textrm{if $\ell\nmid N$} \\
	 1-\ell^{-1}& \textrm{if $\pi_\ell$ is a ramified principal series} \\
	 1-\ell^{-2}& \textrm{if $\pi_\ell$ is unramified special} \\
	 1& \textrm{if $\pi_\ell$ is supercuspidal}. \\
\end{array}\right.
\end{equation}\normalsize

Let $N_0$ be the tame conductor of the residual representation $\ob{\rho}$ of $F$ and let $\Sigma$ be a finite set of primes containing $\{\ell \ \colon\ \ell|(N/N_0)\}$ but not $p$. Denote by $\ob{V}$ (resp. $V_\nu$) the underlying space of $\ob{\rho}$ (resp. of $\rho_\nu$ over $\Frac(\cO_\nu)$), by $\ob{V}_{I_\ell}$ the co-invariants of $\ob{V}$ under the action of the inertia group $I_\ell$ at $\ell$ and put $m_\ell=\dim \ob{V}_{I_\ell}$. Then, the invariance of the Swan conductor under reduction (\cite[\S1]{livne} implies
\begin{equation}\label{eq:swan}
	\cond_\ell(\rho_{\nu})-\cond_\ell(\ob{\rho})=\dim \ob{V}^{I_\ell}-\dim (V_{\nu})^{I_\ell}
\end{equation}
for every $\nu\in\cX(\bI)$ and every prime $\ell\neq p$, where $\cond_\ell$ refers to the local conductor at $\ell$ of Galois representations.
 In particular, $N$ is a divisor of 
\begin{equation}\label{eq:definition_N_Sigma}
	N_\Sigma=N_0\prod_{\ell\in\Sigma} \ell^{m_\ell}.
\end{equation}
\begin{lemma}\label{lem:partition_bad_primes_N_Sigma}
	Let $\Sigma_{sc}$ be the set of supercuspidal primes of $F$ in $\Sigma$.
	 \begin{enumerate}
	 	\item If $F$ is twist-minimal, then $\Sigma=\{\ell \; \colon \; \ell|(N_\Sigma/N)\} \cup \Sigma_{sc}$.
	 	\item Assume $F=F'_\psi$ for some Dirichlet character $\psi$ unramified outside $\Sigma$ and $F'$ a twist-minimal newform. Then  
	 	\[\{\ell \colon \ell|(N_\Sigma/N)\} \cup \Sigma_{sc}=\{\ell\in\Sigma \colon \psi_{(\ell)}=1 \mbox{ or } \psi_{(\ell)}=\theta_{(\ell)} \}\cup \Sigma_{sc}, \]
	 	where the index $(\ell)$ stands for the $\ell$-primary component of a Dirichlet character. Moreover, for any $\ell|(N_\Sigma/N)$, $E_\ell(\ad F)$ and $E_\ell(\ad F')$ are equal up to a unit in $\bI$.
	 \end{enumerate}
\end{lemma}
\begin{proof}
	Note that (1) follows from (2). Let $N'$ and $\theta'$ be the exact level and the character of $F'$ respectively. For $\ell\in\Sigma\setminus S_{sc}$,  \eqref{eq:swan} and the explicit description of the local Galois representation attached to $F'$ at $\ell$ (\cite[Lem. 6.3]{dimitrov2014local}) shows that $\ell$ divides $N_\Sigma/N$ if and only if $N$ and $N'$ share the same $\ell$-adic valuation. This condition is clearly equivalent to $\psi_{(\ell)}=1$ or $\psi_{(\ell)}=(\theta'_{(\ell)})^{-1}=\theta_{(\ell)}$. 
	
	Regarding the last claim, if $\ell|(N_\Sigma/N)$, then $E_\ell(\ad F')=E_\ell(\ad F)$ unless $\ell\nmid N'$, in which case we have $E_\ell(\ad F')=\psi(\ell)^2E_\ell(\ad F)$.
\end{proof}

\begin{theorem}\label{thm:wiles_Sigma}
	Assume $\ob{\rho}$ satisfies \eqref{tag:cr}. Let $\gh'(N_\Sigma p^\infty)$ be the $\Lambda$-subalgebra of $\gh(N_\Sigma p^\infty)$ generated by the Hecke operators $T_n$ for all $n\geq 1$ coprime to $N_\Sigma p$.
	\begin{enumerate}
		\item There exists a unique local component $\bT'_\Sigma=e\gh'(N_\Sigma p^\infty)_{\gm'_\Sigma}$ of $e\gh'(N_\Sigma p^\infty)$ such that there exists an isomorphism $\ob{\rho}\simeq\ob{\rho}_{\gm'_\Sigma}$ respecting the ordinary filtrations.
		\item There exists a unique maximal ideal $\gm_\Sigma\subset e\gh(N_\Sigma p^\infty)$ lifting $\gm'_\Sigma$ and such that $\bT'_\Sigma\simeq\bT_\Sigma$ for $\bT_\Sigma=e\gh(N_\Sigma p^\infty)_{\gm_\Sigma}$. Moreover, $T_\ell\in\gm_\Sigma$ for all $\ell\in\Sigma$ and the image of $T_\ell$ in the localization $\bT_\Sigma$ vanishes. 
	\end{enumerate}
\end{theorem}
\begin{proof}
	This is the content of \cite[Thm. 2.4.1 and Prop. 2.4.2]{emerton2006variation}.
\end{proof}
Denote by $F_\Sigma=\sum_na_n(F_\Sigma)$ the normalized eigenform of $S(N_\Sigma,\bI)$ sharing the same Hecke eigenvalues as $F$ at primes not dividing $N_\Sigma/N$. We have $a_n(F_\Sigma)=0$ for $n$ divisible by a prime in $\Sigma$ and $a_n(F_\Sigma)=a_n(F)$ otherwise. Note that by Theorem \ref{thm:wiles_Sigma} (2), $\bT_\Sigma$ is reduced, so $F_\Sigma$ satisfies \eqref{tag:pr} and it makes sense to consider its adjoint $p$-adic $L$-function.

\begin{proposition}\label{prop:change_of_level_p_adic_L}
	Fix a choice of canonical periods $\xi^\pm_F$ for $F$ as in Definition \ref{def:p_adic_adjoint_L_function}. There is a choice of periods $\xi^\pm_{F_\Sigma}$ for $F_\Sigma$ such that 
	\[L_p(\ad F_\Sigma)=L_p(\ad F)\prod_{\ell|(N_\Sigma/N)}E_\ell(\ad F),\]
	where $L_p(\ad F_\Sigma)$ (resp. $L_p(\ad F)$) is computed with  respect to $\xi^\pm_{F_\Sigma}$ (resp. to $\xi^\pm_{F}$).
\end{proposition}
\begin{proof}
	We keep the notations $\gm^*$, $\bT^*$, $\cM$ and $\lambda^*_{F}$ of \S\ref{sec:interpolation_petersson} and we denote by $\gm^*_\Sigma$, $\bT^*_\Sigma$, $\cM_\Sigma$ and $\lambda^*_{F_\Sigma}$ their counterpart for $F$ replaced with $F_\Sigma$.
	
	First note that $F$ is $\ell$-minimal at any divisor $\ell$ of $N_\Sigma/N$ by Lemma \ref{lem:partition_bad_primes_N_Sigma}, yielding $a_\ell(F)\in\bI^\times$. Write $N_\Sigma/N$ as a product of prime powers $\prod_{i=1}^{d}\ell_i^{e_i}$ with $1\leq e_i \leq 2$ and $\ell_i\in\Sigma$. Then $e_i=2$ only if $\ell_i\nmid N$. We define a $\bI$-linear map 
	\begin{equation}\label{eq:map_varepsilon}
		\varepsilon=\varepsilon_1\circ\ldots\circ\varepsilon_d \colon \cM_\Sigma \to \cM
	\end{equation}
	as follows. Fix $1\leq i \leq d$, and let $\ell=\ell_i$, $e=e_i$ and $M=N \ell_1^{e_1}\ldots \ell_{i-1}^{e_{i-1}}$. At level $Mp^r$, letting $\Gamma=\Gamma_1(Mp^r)$ and $\Gamma'=\Gamma_1(M\ell^e p^r)$ we have a map $\varepsilon_{M,\ell} \colon \HH^1(X_1(M\ell^e p^r),\Zp) \to \HH^1(X_1(Mp^r),\Zp)$ given by 
	\begin{equation*}
		\varepsilon_{M,\ell}(x)= \left\{\begin{array}{ll}
		 x\big|(\left[ \Gamma' \Gamma \right]-\ell^{-1}V_{\ell}^* T_{\ell}^*) & \textrm{if $e=1$} \\
		x\big|(\left[ \Gamma' \Gamma \right]-\ell^{-1} V_{\ell}^* T_{\ell}^*+\ell^{-1} V_{\ell^2}^* \langle \ell \rangle^{-1}) & \textrm{if $e=2$,}
	\end{array}\right.
	\end{equation*}
	where $V_{\ell^m}^*=\left[\Gamma'\begin{pmatrix}1 & 0 \\ 0 & \ell^m\end{pmatrix}\Gamma\right]$. Taking the limit over $r\geq 1$, localizing and extending the scalars to $\bI$, one obtains the map which we denoted $\varepsilon_i$ in \eqref{eq:map_varepsilon}. Under Poincaré duality, $\varepsilon$ coincides with (the extension of scalars of) the map on homology introduced in \cite[\S3.8]{emerton2006variation} and denoted $\varepsilon_\infty$ therein.
	
	 Denote by $(-)^\textbf{t}$ the adjoint with respect to Ohta's pairings \eqref{eq:ohta_pairing} or their specializations \eqref{eq:ota_pairing_interpolation} at tame level $M$. For $\ell=\ell_i$ as before, one finds that
	 \[\varepsilon_{M,\ell}^\textbf{t}(x)=\left\{\begin{array}{ll}
	 	x\big|(V_\ell-\ell^{-1}T_\ell^* \left[ \Gamma \Gamma' \right]) & \textrm{if $e=1$} \\
	 	x\big|(V_{\ell^2}-T_\ell^*V_\ell+ \ell^{-1} \langle \ell \rangle^{-1}\left[ \Gamma \Gamma' \right]) & \textrm{if $e=2$,}
	 \end{array}\right.\]
	where $V_{\ell^m}=\left[\Gamma\begin{pmatrix}\ell^m & 0 \\ 0 & 1\end{pmatrix}\Gamma'\right]$. In particular, we obtain
	\[(\varepsilon\circ\varepsilon^\textbf{t})_{|\cM_{\lambda^*_{F}}}=\prod_{\ell|(N_\Sigma/N)}E'_\ell(\ad F) \id,\]
	where  $E'_{\ell}(\ad F)=-\ell a_{\ell}(F) E_{\ell}(\ad F)$ if $\pi_\ell$ is special or a ramified principal series and $E'_{\ell}(\ad F)=-\ell E_{\ell}(\ad F)$ if $\pi_\ell$ is spherical.
	
	Passing to the $\pm$-part, the map $\varepsilon^\textbf{t}$ injects $\cM_{\lambda^*_{F}}^\pm$ into $(\cM^\pm_{\Sigma})_{\lambda^*_{F_\Sigma}}$. Moreover, \cite[Thm. 3.6.2]{emerton2006variation} implies that $\varepsilon$ is injective on $(\cM^\pm_{\Sigma})_{\lambda^*_{F_\Sigma}}$ so, in fact, $\varepsilon^\textbf{t}(\cM^\pm_{\lambda^*_F})=(\cM^\pm_{\Sigma})_{\lambda^*_{F_\Sigma}}$ by duality. As the element $u=\prod_{\ell|(N_\Sigma/N)}E_\ell(\ad F)/E'_\ell(\ad F)$ belongs to $\bI^\times$, we may take $\xi_{F_\Sigma}^+= u\cdot \varepsilon^\textbf{t}(\xi^+_F)$ (resp.  $\xi_{F_\Sigma}^-=\varepsilon^\textbf{t}(\xi^-_F)$) as a basis of $(\cM^+_{\Sigma})_{\lambda^*_{F_\Sigma}}$ (resp. of $(\cM^-_{\Sigma})_{\lambda^*_{F_\Sigma}}$). 
	Letting $(\cdot,\cdot)_{N,\bI}$ (resp. $(\cdot,\cdot)_{N_\Sigma,\bI}$) be Ohta's pairing in tame level $N$ (resp. tame level $N_\Sigma$), we therefore obtain 
	\begin{align*}
		L_p(\ad F_\Sigma) &= (u\cdot \varepsilon^\textbf{t}(\xi^+_{F}),\varepsilon^\textbf{t}(\xi^-_{F}))_{N_\Sigma,\bI} \\
		&=u\cdot ((\varepsilon\circ\varepsilon^\textbf{t})(\xi^+_F),\xi^-_F)_{N,\bI} \\
		&= L_p(\ad F)\prod_{\ell|(N_\Sigma/N)}E_\ell(\ad F),
	\end{align*}
	as wanted.
\end{proof}

An easy adaptation of the proof of Proposition \ref{prop:change_of_level_p_adic_L} (arguing as in the proof of \cite[Lem. 7.4]{hsiehAJM}) yields the following proposition.

\begin{proposition}\label{prop:twists_and_adjoint_p_adic_L}
	Assume $F\in S(N,\bI)$ is a twist-minimal primitive Hida family, and let $F_\psi$ be the primitive twist of $F$ by a Dirichlet character $\psi$ of conductor $M$ coprime with $p$. Then we can make a choice of canonical periods for $F$ and $F_\psi$ such that
	\[L_p(\ad F_\psi)= L_p(\ad F) \prod_\ell E_\ell(\ad F),\]
	where $\ell$ runs over the set of primes dividing $M$ such that $\psi_{(\ell)}\neq\mathds{1},\theta^{-1}_{(\ell)}$.
\end{proposition}

In the next theorem, we let vary $F$ through the set of Hida families with fixed residual representation $\ob{\rho}$ which are minimally ramified outside a fixed finite set $\Sigma$. 

\begin{theorem}\label{thm:main_2}
	Assume $\ob{\rho}$ satisfies \eqref{tag:cr}. There exists a regular element $L_\Sigma(\ad \ob{\rho})\in\bT_\Sigma$ such that, for all primitive Hida families $F$ with residual representation $\ob{\rho}$ which are minimally ramified outside $\Sigma$, we have 
	\[\lambda_{F_\Sigma}(L_\Sigma(\ad\ob{\rho}))=L_p(\ad F)\prod_{\ell|(N_\Sigma/N)}E_\ell(\ad F)\]
	up to a unit in $\bI$, where $\lambda_{F_\Sigma}\colon \bT_\Sigma \to \bI$ is the system of eigenvalues associated with $F_\Sigma$.
\end{theorem}
\begin{proof}
	 We simply define $L_\Sigma(\ad\ob{\rho})$ as the adjoint $p$-adic $L$-function of Definition \ref{def:formula_adjoint_rho_bar} for $\bT_\Sigma$ and relative to some choice of canonical periods. Any primitive Hida family $F$ of level $N$ with $\ob{\rho}_F\simeq \ob{\rho}$ which is minimally ramified outside $\Sigma$ satisfies $\{\ell|(N/N_0)\}\subset \Sigma$, so Theorem \ref{thm:specialization_Hida_eigenforms} and Proposition \ref{prop:change_of_level_p_adic_L} apply to $F_\Sigma$, yielding
	 \[\lambda_{F_\Sigma}(L_\Sigma(\ad\ob{\rho}))=L_p(\ad F_\Sigma)=L_p(\ad F)\prod_{\ell|(N_\Sigma/N)}E_\ell(\ad F)\]
	 up to units in $\bI$. The fact that $L_\Sigma(\ad\ob{\rho})$ is regular is clear, as $L_p(\ad F_\Sigma)$ does not identically vanish and $\bT_\Sigma$ is reduced.
\end{proof}

\subsection{Selmer groups}
Let $F(q)=\sum_{n\geq 1}a_n(F)q^n\in S(N,\bI)$ be a twist-minimal primitive Hida family of tame level $N$ and character $\theta$, and let $V_F$ be the free $\bI$-module of rank $2$ on which $\rho_F$ acts. We write $\ob{\rho}$ for $\ob{\rho}_F$, that is, $\ob{\rho}=\rho_F \otimes_\bI k$, where $k=\bI/\gm_\bI$ is the residue field of the local $\Lambda$-algebra $\bI$.

Recall that $V_F$ is ordinary in the sense that there exists an exact sequence of free $\bI[G_{\Qp}]$-modules
\[\begin{tikzcd}
	0 \rar & V_F^+ \rar & V_F \rar & V_F^- \rar & 0,
\end{tikzcd}\]
where $G_{\Qp}$ acts on $V_F^-$ via the unramified character $\ur_{a_p(F)}$ sending $\sigma_p$ on $a_p(F)$. 
On the traceless adjoint $V_{\ad}=\Hom^0(V_F,V_F)$ of $V_F$, there is a three-step $G_{\Qp}$-filtration
\begin{equation}\label{eq:filtration}
	0 \subset V_{\ad}^+ \subset V_{\ad}^\dagger \subset V_{\ad},
\end{equation}
where $V_{\ad}^+=\Hom(V_F^-,V_F^+)$ is the subspace of ``upper-triangular'' nilpotent endomorphisms and $V_{\ad}^\dagger$ consists of the space of endomorphisms in $V_{\ad}$ fixing $V_F^+$. We put $V_{\ad}^-=V_{\ad}/V_{\ad}^+$. 

Let $M^\vee=\Hom_{\Zp}(M,\Qp/\Zp)$ for any profinite $\Zp$-module $M$ and let $\Sigma$ be a set of finite places of $\bQ$ not containing $p$. We define Selmer groups à la Greenberg, using the discrete modules $D_{\ad}=V_{\ad} \otimes_{\bI} \bI^\vee$ and $D^\pm_{\ad}=V^\pm_{\ad} \otimes_{\bI} \bI^\vee$, by letting
\small
\begin{equation}\label{eq:def_sel_sigma}
	\Sel^\Sigma(\ad F)= \ker\left[\HH^1(\bQ,D_{\ad}) \to \prod_{\ell \notin \Sigma\cup\{p\}} \HH^1(I_\ell,D_{\ad}) \times \HH^1(I_p,D^-_{\ad})\right],
\end{equation}\normalsize
and $\Sel(\ad F)=\Sel^\emptyset(\ad F)$. Standard arguments show that $\Sel^\Sigma(\ad F)^\vee$ is a finitely generated $\bI$-module. 

\begin{theorem}\label{thm:main_conjecture_large_sigma}
	Assume \eqref{tag:cr'} holds for $\ob{\rho}$. Then $\Sel^\Sigma(\ad F)^\vee$ is a finitely generated torsion $\bI$-module. If $\Sigma$ contains all the places dividing $N$, then 
	\[\char_\bI \Sel^\Sigma(\ad F)^\vee=L_p(\ad F_\Sigma)\cdot \bI.\]
\end{theorem}
\begin{proof}
We may assume in the the first claim that $\Sigma$ contains the primes dividing $N$. As $L_p(\ad F_\Sigma)$ is not identically $0$ (it does not vanish at any $\nu\in\cX(\bI)$ such that $F_\nu=F_\nu^\circ$ by Thm. \ref{thm:calcul_valeurs_speciales}), the torsionness of $\Sel^\Sigma(\ad F)^\vee$ will follow from the second claim which we now prove.

Let $\cO=W(k)$, $\Lambda_\cO=\cO[[U_1]]$ and let $\bT_\Sigma$ be the local component associated with $\ob{\rho}$ and $\Sigma$ as in Theorem \ref{thm:wiles_Sigma}. Then $F_\Sigma$ gives rise to a homomorphism of $\Lambda_\cO$-algebras $\lambda_\Sigma \colon \bT_\Sigma \to \bI$. 

Consider the following deformation problem. Let $\Phi$ be the functor classifying $p$-ordinary lifts of $\ob{\rho}$ unramified outside $S:=\Sigma\cup\{p\}$ modulo strict equivalence. 
More precisely, let $\bQ_S\subset \ob{\bQ}$ be the maximal extension of $\bQ$ unramified outside $S$ and let $\ob{\rho}^\pm$ be the $k$-valued $G_{\Qp}$-representation acting on $V_F^\pm\otimes_\bI k$.
For any local complete noetherian $\cO$-algebra $A$ with residue field $k$, $\Phi(A)$ consists of classes $[\rho]_\sim$, where $\rho \colon G_{\bQ,S}=\Gal(\bQ_S/\bQ) \to \GL_2(A)$ is such that $\rho \otimes_A k\simeq \ob{\rho}$ and such that there exists a $G_\Qp$-equivariant exact sequence of free $A$-modules $0 \to \rho^+ \to \rho \to \rho^-\to 0 $ with $\rho^\pm \otimes_{A} k\simeq \ob{\rho}^\pm$ and $\rho^-$ unramified. The equivalence relation $\sim$ here means that $\rho\sim\rho'$ if $\rho'$ can be obtained from $\rho$ by conjugation by a matrix in $\ker(\GL_2(A)\to\GL_2(k))$. Since $\ob{\rho}$ is absolutely irreducible, $\Phi$ is representable by a universal pair $(\rho_\Sigma,R_\Sigma)$, $[\rho_\Sigma]_\sim\in\Phi(R_\Sigma)$ by Mazur's theory \cite{mazurfermatbook}. 
Under our running hypotheses on $\ob{\rho}$, by the work of Wiles, Taylor-Wiles and Diamond \cite{wilesfermat,taylorwiles,diamond1996} (see also \cite[p. 539]{emerton2006variation}), we then have an isomorphism 
\begin{equation}\label{eq:R=T}
	R_\Sigma\simeq\bT_\Sigma.
\end{equation}
In fact, $\Lambda_\cO$ can be interpreted as the universal deformation ring associated with $\det(\ob{\rho})$, so $R_\Sigma$ comes equipped with a structure of $\Lambda_\cO$-algebra, and \eqref{eq:R=T} is an isomorphism of $\Lambda_\cO$-algebras. A deep consequence of \eqref{eq:R=T} is that $\bT_\Sigma$ is a local complete intersection $\Lambda_\cO$-algebra. This provides us, by a theorem of Tate \cite[Thm. 6.8]{hidapune}, with an equality 
\[\char_\bI C_0(\lambda_\Sigma)=\char_\bI C_1(\lambda_\Sigma),\]
where $C_1(\lambda_\Sigma)=\Omega_{\bT_\Sigma/\Lambda_\cO}\otimes_{\bT_\Sigma,\lambda_\Sigma} \bI$ (and the symbol $\Omega$ here refers to continuous Kähler differentials). In light of Proposition \ref{prop:congruence_module_and_p_adic_L_functions}, it is enough to show that $C_1(\lambda_\Sigma)^\vee\simeq\Sel^\Sigma(\ad F)$ as $\bI$-modules, which is a standard fact from Mazur's theory (see e.g. \cite[\S6]{mazurtilouine} and \cite[\S30]{mazurfermatbook}), but we explain the main lines for sake of completeness. We have:
\begin{align*}
	C_1(\lambda_\Sigma)^\vee=\Hom_{\Zp}(\Omega_{\bT_\Sigma/\Lambda_\cO}\otimes_{\bT_\Sigma,\lambda_\Sigma} \bI,\Qp/\Zp)
	&\simeq\Hom_{\bT_\Sigma}(\Omega_{\bT_\Sigma/\Lambda_\cO},\bI^\vee) \\
	&\simeq\Der_{\Lambda_\cO}(\bT_\Sigma,\bI^\vee),
\end{align*}
the last term referring to the module of $\Lambda_\cO$-linear continuous derivations (and $\bI^\vee$ being seen as a discrete $\bT_\Sigma$-module via $\lambda_\Sigma$). For $n\geq 1$, let $A_n=(\bI/\gm_\bI^n)^\vee$ which we see as a $\bT_\Sigma$-module, and note that $\Der_{\Lambda_\cO}(\bT_\Sigma,\bI^\vee)=\cup_{n\geq 1}\Der_{\Lambda_\cO}(\bT_\Sigma,A_n)$. The free $\bI$-module $\bI[A_n]$ generated by elements of $A_n$ also has a structure of $\bT_\Sigma$-algebra if we let $t\cdot a=\lambda_\Sigma(t)a$ and $a\cdot b=0$ for all $t\in\bT_\Sigma$ and $a,b\in A_n$. There is a natural injective map 
\begin{equation}\label{eq:der_to_hom}
	\Der_{\Lambda_\cO}(\bT_\Sigma,A_n) \hookrightarrow \Hom_{\Lambda_\cO}(\bT_\Sigma,\bI[A_n])
\end{equation}
sending a derivation $\diff\colon \bT_\Sigma \to A_n$ to $\phi_{\diff}=\lambda_\Sigma+\iota\circ \diff$, where $\iota\colon A_n \hookrightarrow\bI[A_n]$ is the natural injection. Further, the map \eqref{eq:der_to_hom} has image $\cH_n:=\{\phi\colon \phi\equiv \lambda_\Sigma \mod A_n\}\subseteq \Hom_{\Lambda_\cO}(\bT_\Sigma,\bI[A_n])$. By the universal property of $\bT_\Sigma$, it holds that
\[\cH_n=\left\{[\rho]_\sim\in\Phi(\bI[A_n]) \ : \ \rho\equiv\rho_F \mod A_n\right\}.\]
Now, given a representative $\rho$ of a class $[\rho]_\sim\in\Phi(\bI[A_n])$, the condition $\rho\equiv \rho_F$ mod $A_n$ shows that the map $c\colon G_{\bQ,S} \to \End(V_F)\otimes_{\bI} A_n$ defined by the relation $\rho=\rho_F\cdot(1+c)$ is a cocycle in $\HH^1(\bQ_S/\bQ,V_{\ad}\otimes_{\bI}A_n)$. Moreover, $c$ is replaced by a cohomologous cocycle if $\rho$ is replaced by a representation strict-equivalent to it. Therefore, the map
\[\cH_n \to \HH^1(\bQ_S/\bQ,V_{\ad}\otimes_{\bI}A_n)\]
induced by $\rho \mapsto c_\rho=\rho\cdot\rho_F^{-1}-1$ is well-defined and injective, and its image consists of classes of cocycles $c$ such that $\rho=\rho_F\cdot(1+c)$ is $p$-ordinary. We check that this condition on $c_\rho$ amounts to $c_\rho(I_p)\subseteq V_{\ad}^+\otimes A_n$, that is, $c_\rho(I_p)$ consists of upper-nilpotent matrices in an ordinary basis for $\rho$.

If $\rho$ is $p$-ordinary, then $\rho_{|I_p}$ and $\rho_{F|I_p}$ look like $\left(\begin{smallmatrix}* & * \\ 0 & 1\end{smallmatrix}\right)$ in their common ordinary basis, and hence, it is clear that $c_\rho(I_p)\subseteq \left(\begin{smallmatrix}* & * \\ 0 & 0\end{smallmatrix}\right)\cap (V_{\ad}\otimes A_n)=V^+_{\ad}\otimes A_n$. Conversely, we obtain a $p$-ordinary basis for $\rho$ by tensoring that of $\rho_F$ by $\bI[A_n]$ over $\bI$. This proves that 
\[\cH_n\simeq \ker\left[\HH^1(\bQ_S/\bQ,V_{\ad}\otimes_{\bI} (\bI/\gm_\bI^n)^\vee) \to \HH^1(I_p,V_{\ad}^-\otimes_{\bI}(\bI/\gm_\bI^n)^\vee))\right] \]
for all $n$, and by taking the direct limit over $n\geq 1$, one sees that $C_1(\lambda_\Sigma)^\vee\simeq\varinjlim_n\cH_n\simeq\Sel^\Sigma(\ad F)$ as claimed.
\end{proof}

The aim of the rest of the section is to relax the assumption on $\Sigma$ in Thm. \ref{thm:main_conjecture_large_sigma}.
We introduce more notations and we let, for any prime $\ell$, 
\begin{equation}\label{eq:euler_factor_selmer_group}
	\HH^1_{\Gr}(\bQ_\ell,D_{\ad})=\left\{\begin{array}{ll}
		\ker\left[\HH^1(\bQ_\ell,D_{\ad}) \to \HH^1(I_\ell,D_{\ad})\right] & \mbox{if $\ell\neq p$} \\ 
		\ker\left[\HH^1(\bQ_p,D_{\ad}) \to \HH^1(I_p,D^-_{\ad})\right] & \mbox{if $\ell= p$,}
	\end{array}\right.
\end{equation}
and we put 
\begin{equation}
	\HH^1_{\slash\Gr}(\bQ_\ell,D_{\ad})=\dfrac{\HH^1(\bQ_\ell,D_{\ad})}{\HH^1_{\Gr}(\bQ_\ell,D_{\ad})}.
\end{equation}
In particular, for any finite sets of primes $\Theta\subseteq \Sigma$ not containing $p$, we have an exact sequence
\[\begin{tikzcd}
	0 \rar & \Sel^{\Theta}(\ad F) \rar & \Sel^\Sigma(\ad F) \rar & \prod_{\ell\in \Sigma-\Theta} \HH^1_{\slash\Gr}(\bQ_\ell,D_{\ad}).
\end{tikzcd}\]
\begin{proposition}\label{prop:calcul_euler_factors_selmer}
Let $\ell \neq p$ be a rational prime. Then $\HH^1_{\slash\Gr}(\bQ_\ell,D_{\ad})$ is of $\bI$-cotorsion. Moreover, $\char_{\bI} \HH^1_{\slash\Gr}(\bQ_\ell,D_{\ad})^\vee=E_\ell(\ad F)\cdot\bI$, where $E_\ell(\ad F)$ is as in \eqref{eq:euler_factor}, unless $\ell$ belongs to the set $\Sigma_e$ defined in \eqref{eq:definition_exceptional_primes}, in which case $\char_{\bI} \HH^1_{\slash\Gr}(\bQ_\ell,D_{\ad})^\vee=(1+\ell)^{-1}\cdot\bI$.
\end{proposition}
\begin{proof}
	Let $G_\ell=G_{\bQ_\ell}$ and let $\sigma_\ell\in G_\ell/I_\ell$ be the arithmetic Frobenius. By the inflation-restriction exact sequence, $\HH^1_{\slash\Gr}(\bQ_\ell,D_{\ad})\simeq \HH^1(I_\ell,D_{\ad})^{G_\ell/I_\ell}$ as $\bI$-modules. Let $\psi\colon G_\ell \to \bI^\times$ be a continuous character and let $D=\bI^\vee(\psi)$ be the cofree $\bI$-module of rank one endowed with a $G_\ell$-action given by $\psi$. We claim that 
	\[\cH:=\HH^1(I_\ell,D)^{G_\ell/I_\ell}=\left\{\begin{array}{ll}
		(\bI/(1-\psi(\sigma_\ell)\ell^{-1})\bI)^\vee & \mbox{if $\psi$ is unramified,} \\ 
		0 & \mbox{if $\psi$ is ramified.}
	\end{array}\right.\]
	Recall that $I_\ell$ fits into an exact sequence of $G_\ell/I_\ell$-modules
	\[\begin{tikzcd}
			0 \rar & J_\ell \rar & I_\ell \rar & \Zp(1) \rar & 0
	\end{tikzcd}\]
	with $J_\ell$ of prime-to-$p$ order. This property of $J_\ell$ shows that the inflation map induces an isomorphism $\HH^1(I_\ell/J_\ell,D^{J_\ell})\simeq \HH^1(I_\ell,D)$, and moreover $D^{J_\ell}$ is $\bI$-divisible. Since $D$ is $\bI$-cofree of corank $1$, either $D^{J_\ell}=0$ or $D^{J_\ell}=D$ holds. In the first case, we then have $\cH=0$ and $\psi$ is ramified, so we may assume $D^{J_\ell}=D$. This means that $\psi_{|I_\ell}$ factors through $I_\ell/J_\ell$, which is procyclic, so $\HH^1(I_\ell/J_\ell,D^{J_\ell})\simeq D/(t-1)D$, where $t=\psi(g)$ for a topological generator $g$ of $I_\ell/J_\ell$. Of course, $t\neq 1$ if $\psi$ is ramified, in which case $\cH\hookrightarrow D/(t-1)D=0$ as $D$ is divisible. In the unramified case, one has
	\small\begin{align*}\cH=\Hom(I_\ell,D)^{G_\ell/I_\ell}\simeq\Hom(\Zp(1),D)^{G_\ell/I_\ell}=&D(-1)^{G_\ell/I_\ell}\\ 
		=&D[1-\psi(\sigma_\ell)\ell^{-1}] \\ 
		=& (\bI/(1-\psi(\sigma_\ell)\ell^{-1})\bI)^\vee,
	\end{align*}\normalsize
	as claimed. 
	
	Now we turn back to the proof of the lemma. If $\ell\nmid N$, then 
	\begin{equation}\label{eq:proof_lemma_euler_factors}
		D_{\ad}\simeq \bI^\vee(\chi_1\chi_2^{-1})\oplus \bI^\vee \oplus \bI^\vee(\chi_2\chi_1^{-1})
	\end{equation}
	as $\bI[G_\ell]$-modules, where $\chi_1,\chi_2$ are the unramified characters such that $1-a_\ell(F)X+\theta_\bI(\ell)\ell X^2=(1-\chi_1(\sigma_\ell)X)(1-\chi_2(\sigma_\ell)X)$. Up to a unit of $\bI$, $E_\ell(\ad F)$ is equal to $(1-\chi_1\chi_2^{-1}(\sigma_\ell)\ell^{-1})(1-\ell^{-1})(1-\chi_2\chi_1^{-1}(\sigma_\ell)\ell^{-1})$, so the first part of the proof shows that $\char_{\bI} \HH^1_{\slash\Gr}(\bQ_\ell,D_{\ad})^\vee=E_\ell(\ad F)\cdot\bI$, as wanted.
	
	In the case where every classical specialization of $F$ is a ramified principal series at $\ell$, a similar isomorphism as in \eqref{eq:proof_lemma_euler_factors} with $\chi_1$ ramified and $\chi_2$ unramified holds. In particular, both $\chi_1\chi_2^{-1}$ and $\chi_2\chi_1^{-1}$ are ramified, so $\char_{\bI} \HH^1_{\slash\Gr}(\bQ_\ell,D_{\ad})^\vee=(1-\ell^{-1})\cdot \bI=E_\ell(\ad F)\cdot\bI$.
	
	Assume now that $F$ is a special at $\ell$. By \cite[Lem. 6.3]{dimitrov2014local}, one has $\rho_{F|G_\ell}\sim \begin{pmatrix}\psi & * \\ 0 & 1\end{pmatrix}$, where $\psi$ is the unramified character of $G_\ell$ sending $\sigma_\ell$ to $\ell$ and $*$ is nontrivial on $I_\ell$. In particular, $\rho_F(J_\ell)=\{1\}$ and $\rho_F(I_\ell)=\begin{pmatrix}1 & \Zp.a \\ 0 & 1\end{pmatrix}$ for some $a\in \bI\setminus\{0\}$. Let $g\in I_\ell$ be such that $\rho_F(g)=\begin{pmatrix}1 & a \\ 0 & 1\end{pmatrix}$ and consider the adjoint action of $g$ on $V_{\ad}$. In the basis $\{e_1,e_2,e_3\}=\{\begin{pmatrix}0 & 1 \\ 0 & 0\end{pmatrix},\begin{pmatrix}1 & 0 \\ 0 & -1\end{pmatrix},\begin{pmatrix}0 & 0 \\ 1 & 0\end{pmatrix}\}$, the matrix of $\ad \rho_F(g)$ is 
	\[M=\begin{pmatrix}1 & -2a & -a^2 \\ 0 & 1 & a \\ 0 & 0& 1\end{pmatrix}.\]
	Hence,
	\small
\[		\HH^1(I_\ell,D_{\ad})=(D_{\ad})_{I_\ell}=D_{\ad}/(M-\id)D_{\ad} =\dfrac{\bI^\vee.e_1 \oplus \bI^\vee.e_2 \oplus\bI^\vee.e_3}{a\bI^\vee.e_1\oplus a\bI^\vee.e_2\oplus 0.e_3} =\bI^\vee.e_3,\\		
\]
	\normalsize
	as $M-\id$ has the same image as that of  $\left(\begin{smallmatrix}
		0 & a & 0 \\ 0 & 0 & a \\ 0 & 0& 0
	\end{smallmatrix}\right)$ 
	as an endomorphism of $V_{\ad}$. Finally, the action of $\ad \rho_F$ on $D_{\ad}/(\bI^\vee.e_1 \oplus \bI^\vee.e_2)$ is given by $\psi^{-1}$, so we obtain 
	\small\[\HH^1(I_\ell,D_{\ad})^{G_\ell/I_\ell}= \HH^1(I_\ell,\bI^\vee(\psi^{-1}))^{G_\ell/I_\ell}=(\bI/(1-\ell^{-2})\bI)^\vee=(\bI/E_\ell(\ad F)\bI)^\vee.\]\normalsize
	Consider now the last case where every specialization of $F$ is supercuspidal at $\ell$. First notice that, as $D_{\ad}^{J_\ell}\simeq (D_{\ad})_{J_\ell}$ canonically, $\HH^1(I_\ell,D_{\ad})\simeq \HH^1(I_\ell/J_\ell,D_{\ad}^{J_\ell})\simeq (D_{\ad})_{I_\ell}$, which is $\bI$-divisible.
	Assume now $\ell\notin \Sigma_e$. For any $\nu\in\cX(\bI)$, the local $L$-factor at $\ell$ of the adjoint lift of $F^\circ_\nu$ is $1$ by \cite[Corollary (1.3)]{gelbartjacquet}. Thus, $(V_{\ad})_{I_\ell}=0$ by the local Langlands correspondence and Nakayama's lemma. It follows that $(D_{\ad})_{I_\ell}$ is cotorsion, hence trivial, and $\HH^1_{/\Gr}(\bQ_\ell,D_{\ad})$ has characteristic ideal generated by $E_\ell(\ad F)=1$. 
	
	Assume finally $\ell\in\Sigma_e$ and let $G_{\ell^2}=\Gal(\ob{\bQ}_\ell/\bQ_{\ell^2})$. Then $\rho_{F|G_\ell} \simeq \Ind_{G_{\ell^2}}^{G_\ell} \chi$ for some ramified character $\chi$ of $G_{\ell^2}$. In particular, $\ad \rho_{F|G_\ell}\simeq \Ind_{G_{\ell^2}}^{G_\ell}\chi/\chi' \oplus \eta$, where $\chi'$ is the conjugate of $\chi$ under $\Gal(\bQ_{\ell^2}/\bQ_\ell)$ and $\eta \colon \Gal(\bQ_{\ell^2}/\bQ_\ell)\simeq\{\pm1\}$ is the unramified quadratic character. Note that $\chi/\chi'$ is ramified, as every classical specialization of $\chi/\chi'$ is. Using that $I_\ell \subset G_{\ell^2}$, we obtain $\HH^1_{/\Gr}(\bQ_\ell,D_{\ad})=\HH^1_{/\Gr}(\bQ_\ell,\bI^\vee(\eta))$, hence $\char_{\bI}\HH^1_{/\Gr}(\bQ_\ell,D_{\ad})=(1-\eta(\sigma_\ell)\ell^{-1}).\bI=(1+\ell^{-1}).\bI$.
\end{proof}

\begin{proposition}\label{prop:surjectivity_phi}
	Let $S=\Sigma\cup\{p,\infty\}$ and let $\bQ_S$ be the maximal extension of $\bQ$ unramified outside $S$. The localization map 
	\[\phi \colon \HH^1(\bQ_S/\bQ,D_{\ad}) \to \prod_{\ell\in\Sigma\cup\{p\}} \HH^1_{\slash\Gr}(\bQ_\ell,D_{\ad}) \]
	is surjective.
\end{proposition}
\begin{proof}
Without loss of generality, we may assume that $\Sigma$ contains all the places dividing $N$. We apply Greenberg's result \cite[Prop. 2.6.3]{greenberg2015structure} which proves the surjectivity of $\phi$ under certain hypotheses that we now check. By assumption, $\rho_{F}$ is residually irreducible when restricted to $G_{\bQ(\mu_p)}$, so $\ob{V}_{\ad}=D_{\ad}[\gm_\bI]$ has trivial $G_{\bQ(\mu_p)}$-invariants by Schur's lemma. Since $\ob{V}_{\ad}$ is self-dual, $(\ob{V}_{\ad})_{G_{\bQ(\mu_p)}}=0$, and in particular, $\ob{V}_{\ad}$ has no quotient isomorphic to $\mu_p$. Moreover, $\bI$ being an integrally closed domain which is finite over $\Lambda$, it is reflexive as a $\Lambda$-module. But $\Lambda$ has Krull dimension $2$, so $\bI$ and $T_{\ad}$ are $\Lambda$-free. This checks the validity of (b) of \textit{loc. cit.}.

We now check the weak Leopoldt condition (LEO) and the cotorsionness of $\coker\phi$ (CRK). For (LEO), it suffices to prove that $\HH^2(\bQ_S/\bQ,D_{\ad})$ is of $\bI$-cotorsion. Since $\rho_F$ is odd, $\corank_{\bI}((D_{\ad})^{G_{\bR}})=1$, and the computation of global Euler characteristics (\cite[\S2.3]{greenberg2015structure}) yields 
\[ \corank_{\bI}(\HH^1(\bQ_S/\bQ,D_{\ad}))=\corank_{\bI}(\HH^2(\bQ_S/\bQ,D_{\ad}))+2.\]
Note that $\ker \phi=\Sel(\ad F)$ is cotorsion by Thm. \ref{thm:main_conjecture_large_sigma}. Hence, if we let $C$ be the codomain of $\phi$, then the equality $\corank_{\bI}(C)= 2$, which we show now, will imply (LEO) and (CRK). By Proposition \ref{prop:calcul_euler_factors_selmer}, we have to show that $\corank_{\bI}(\HH^1_{\slash\Gr}(\Qp,D_{\ad}))=2$. 

Recall the filtration \eqref{eq:filtration} which comes from the ordinary filtration on $V_F$. Letting $\chi_1,\chi_2\colon G_{\Qp} \to \bI^\times$ with $\chi_2$ unramified be such that $\rho_{F|G_{\Qp}}$ is an extension of $\chi_2$ by $\chi_1$, we have short exact sequences of $\bI[G_{\Qp}]$-modules
\[\begin{tikzcd}
	0 \rar & \bI^\vee(\chi_1\chi_2^{-1}) \rar & D_{\ad} \rar & D_{\ad}^- \rar & 0, 
\end{tikzcd}\]
and
\[\begin{tikzcd}
	0 \rar & \bI^\vee \rar & D_{\ad}^- \rar & \bI^\vee(\chi_2\chi_1^{-1}) \rar & 0,
\end{tikzcd}\]
From the explicit description of $\det(\rho_F)$, one sees that $\bI(\chi_1\chi_2^{-1})$ is not isomorphic to $\bI(n)$ for any $n\in\bZ$. Thus, the cokernel of the map $\HH^1(\Qp,D_{\ad}) \to \HH^1(\Qp,D_{\ad}^-)$ is of cotorsion, as $\HH^2(\Qp,\bI^\vee(\chi_1\chi_2^{-1}))\simeq\bI^\vee(\chi_2\chi_1^{-1})(1)^{G_{\Qp}}$ is itself cotorsion. This shows that $\HH^1_{\slash\Gr}(\Qp,D_{\ad})$ has the same corank as $\im(\HH^1(\Qp,D_{\ad^-})\to\HH^1(I_p,D^-_{\ad}))=\HH^1(I_p,D^-_{\ad})^{G_{\Qp}/I_p}$. In particular,
\small\[\corank_{\bI}(\HH^1_{\slash\Gr}(\Qp,D_{\ad}))=\corank_{\bI}(\HH^1(\bQ_p,D^-_{\ad}))-\corank_{\bI}(\HH^1_\ur(G_{\Qp},D^-_{\ad})),\]\normalsize
where $\HH^1_\ur(G_{\Qp},D^-_{\ad})=\HH^1(G_{\Qp}/I_p,(D^-_{\ad})^{I_p})$. Using the second exact sequence, one finds $\HH^1_\ur(G_{\Qp},D^-_{\ad})$ has the same corank as $\HH^1(G_{\Qp}/I_p,\bI^\vee)$ which is $1$, whereas the Euler-Poincaré formula yields
\small
\begin{align*}
\corank_{\bI}(\HH^1(\bQ_p,D^-_{\ad}))&=\corank_{\bI}((D^-_{\ad})^{G_{\Qp}})+\corank_{\bI}(\HH^2(\Qp,D^-_{\ad}))+2 \\ &=1+0+2=3.
\end{align*}
\normalsize
This proves that $\corank_{\bI}(\HH^1_{\slash\Gr}(\Qp,D_{\ad}))=2$, and that $\phi$ is surjective.
\end{proof}

\begin{theorem}\label{thm:main_3}
	Assume $\ob{\rho}$ satisfies \eqref{tag:cr'} and that $F$ is twist-minimal. Then the following equality of $\bI$-ideals holds:
	\[\char_\bI \Sel(\ad F)^\vee = \prod_{\ell\in\Sigma_e}(1+\ell^{-1})^{-1}L_p(\ad F)\cdot \bI.\]
\end{theorem}
\begin{proof}
	Choose a finite set $\Sigma$ containing the primes dividing $N$ (but not $p$), so $\Sigma_e\subseteq \Sigma$. Note that $\prod_{\ell\in \Sigma}E_\ell(\ad F)=\prod_{\ell|(N_\Sigma/N)}E_\ell(\ad F)$ by Lemma \ref{lem:partition_bad_primes_N_Sigma} (1). Proposition \ref{prop:surjectivity_phi} provides us with an exact sequence
	\small\[\begin{tikzcd}
		0 \rar & \Sel(\ad F) \rar &   \Sel^\Sigma(\ad F) \rar & \prod_{\ell\in \Sigma} \HH^1_{\slash\Gr}(\bQ_\ell,D_{\ad}) \rar & 0.
	\end{tikzcd}\]\normalsize
	The multiplicativity of characteristic ideals with exact sequences, together with Proposition \ref{prop:calcul_euler_factors_selmer}, Theorem \ref{thm:main_conjecture_large_sigma} and Proposition \ref{prop:change_of_level_p_adic_L}, then yields:
	\begin{align*}
		\char_{\bI}\Sel(\ad F)^\vee   &= \char_{\bI}\Sel^\Sigma(\ad F)^\vee\cdot \prod_{\ell\in\Sigma_e}(1+\ell^{-1})^{-1} \prod_{\ell\in \Sigma}E_\ell(\ad F)^{-1} \\
										&= L_p(\ad F_\Sigma)\cdot \prod_{\ell\in\Sigma_e}(1+\ell^{-1})^{-1}\prod_{\ell\in \Sigma}E_\ell(\ad F)^{-1}\cdot \bI \\
										&=\prod_{\ell\in\Sigma_e}(1+\ell^{-1})^{-1}L_p(\ad F)\cdot \bI, 
	\end{align*}
	as wanted.
\end{proof}


\appendix
\section{Congruence modules}\label{appendix}
We recall some standard definitions and results on congruence modules in a slightly different setting than in \cite[\S2.7.5]{eigenbook}, \cite[\S6.2.2]{hidapune}, or in \cite[\S2]{tilouineurbanANT}. Let $\Lambda$ be a local domain, and let $\cT$ be a commutative ring containing $\Lambda$. We further assume that $\cT$ is a finite flat $\Lambda$-algebra. We put $K=\Frac(\Lambda)$ and we write $M_K=M \otimes_\Lambda K$ for any $\Lambda$-module $M$. Assume that we are given 
\begin{itemize}
	\item $M,N$ two free $\cT$-modules of rank $1$, and
	\item $\lambda \colon \cT \twoheadrightarrow \Lambda$ a surjective $\Lambda$-homomorphism. 
\end{itemize}
Assume further that we have a decomposition as $K$-algebras of $\cT_K$ as $K \times X$, where the projection of $\cT_K$ onto the first factor is induced by $\lambda$. We let $e_1,e_2$ be the corresponding idempotents of $\cT_K$, and we let
\[M^\lambda=e_1 M \subset e_1 M_K \subset M_K, \qquad M_\lambda=M \cap e_1 M_K \subset e_1 M_K \subset M_K,\]
and similarly for $N$. Note that $M_\lambda \subseteq M^\lambda$, as $m=e_1\cdot m \in M^\lambda$ for any $m\in M_\lambda$. We further define the congruence module associated to $M$ and $\lambda$ as
\[C_0^M(\lambda)=M^\lambda/M_\lambda.\]
The congruence ideal associated to $M$ and $\lambda$ is the annihilator of $C^M_0(\lambda)$ as a $\Lambda$-module.
Letting $C_0(\lambda):=C_0^\cT(\lambda)$, we then have $C_0^M(\lambda) \simeq C_0(\lambda)$, as $M\simeq \cT$ as $\cT$-modules.

\begin{lemmaAppendix}\label{lem:appendix_direct_summand}
	$M_\lambda$ is $\Lambda$-free of rank $1$, and it is a direct summand of $M$. 
\end{lemmaAppendix}
\begin{proof}
Fix an isomorphism $M\simeq\cT$. We see $\cT$ as a sub-$\Lambda$-module of $\cT_K=e_1\cT_K\oplus e_2\cT_K$. The projection $\cT_K \twoheadrightarrow e_1 \cT_K$ is $\lambda\otimes K$, so we have $\cT_\lambda=\im(\cT \to e_1 \cT_K)=\im(\lambda)=\Lambda$ as $\Lambda$-modules. Hence, $M_\lambda$ is $\Lambda$-free of rank $1$.

Since $\lambda_{|\Lambda}=\id$, we have $\cT = \im(\lambda) \oplus \ker(\lambda)$, \textit{i.e.}, $\cT=\cT_\lambda \oplus (\cT\cap e_2\cT_K)$ inside $\cT_K$. Therefore, we obtain $M=M_\lambda \oplus (M\cap e_2 M_K)$.
\end{proof}

Suppose we are given a $\Lambda$-linear pairing $(\cdot,\cdot) \colon M\times N \to \Lambda$ which is perfect (\textit{i.e.}, it induces  isomorphisms $N \simeq \Hom_\Lambda(M,\Lambda)$ and $M \simeq \Hom_\Lambda(N,\Lambda)$) and such that 
\begin{equation}\label{eq:appendix_condition_T_equivariance_pairing}
	\forall t \in \cT,\ \forall (m,n)\in M\times N, \qquad (tm,n)=(m,tn).
\end{equation}
Note that the orthogonal of $e_1M_K$ (resp. of $e_2M_K$) is $e_2N_K$ (resp. $e_1N_K$). As a last piece of notation, we define for any $M'\subseteq M_K$ 
\[(M')^*:=\{n\in N_K \ \colon\ \forall m\in M',\  (m,n)\in\Lambda \},\]
and we define similarly $(N')^*$ for $N'\subseteq N_K$.

\begin{propAppendix}\label{prop:appendix_orthogonality}
	We have $(M_\lambda)^*=N^\lambda \oplus e_2 N_K$, and $(M^\lambda)^*=N_\lambda \oplus e_2 N_K$. Moreover, $(\cdot,\cdot)$ induces an isomorphism $M^\lambda\simeq \Hom_\Lambda(N_\lambda,\Lambda)$ (and also $N^\lambda\simeq \Hom_\Lambda(M_\lambda,\Lambda)$).
\end{propAppendix}
\begin{proof}
	We first check that $(M_\lambda)^*\supseteq N^\lambda \oplus e_2 N_K$. Take $n_1\in N$, $n_2\in N_K$ and let $n=e_1n_1+e_2n_2$. For any $m\in M_\lambda$, we have $m=e_1m$, so $(m,e_2n_2)=(m,e_1e_2n_2)=0$ by \eqref{eq:appendix_condition_T_equivariance_pairing}. Hence,
	\[(m,n)= (m,e_1n_1) = (e_1m,n_1)=(m,n_1)\in\Lambda, \]
	as $m\in M$ and $n_1\in N$. This shows that $n\in (M_\lambda)^*$. 
	
	We now prove $(M_\lambda)^*\subseteq N^\lambda \oplus e_2 N_K$ and we take $n\in (M_\lambda)^*$. The map $m \mapsto (m,n)$ defines a $\Lambda$-linear form $M_\lambda \to \Lambda$, which can be extended to a linear form $M \to \Lambda$ by Lemma \ref{lem:appendix_direct_summand}. By perfectness of $(\cdot,\cdot)$, there exists $p\in N$ such that $(\cdot,n)=(\cdot,p)_{|M_\lambda}$. Moreover, as $\Lambda \simeq M_\lambda \subset e_1 M_K\simeq K$ by Lemma \ref{lem:appendix_direct_summand}, we have $M_\lambda \otimes_\Lambda K=e_1 M_K$. In particular, $n-p$ is orthogonal to $e_1M_K$ by linearity of the pairing, that is, $n-p\in e_2N_K$. Therefore, $n\in e_1p+e_2N_K \subset N^\lambda \oplus e_2N_K$, as wanted.
	
	Since $e_1N_\lambda=N_\lambda\subset N$, we have (as sets) $(M^\lambda,N_\lambda \oplus e_2 N_K)=(M^\lambda,N_\lambda)=(e_1M,N_\lambda)=(M,N_\lambda)\subset \Lambda$, which implies $(M^\lambda)^*\supseteq N_\lambda \oplus e_2 N_K$. We check the reverse inclusion and we let $n\in (M^\lambda)^*$. To prove that $e_1n\in N$, we remark that, for all $m\in M$, we have $(m,e_1n)=(e_1m,n)\in\Lambda$, so the map $m\mapsto (m,e_1n)$ belongs to $\Hom_\Lambda(M,\Lambda)$. By the perfectness of $(\cdot,\cdot)$, this means that $e_1n\in N$, as claimed.
	
	It remains to prove that $m\mapsto (n\mapsto (m,n))$ yields an isomorphism $M^\lambda\simeq \Hom_\Lambda(N_\lambda,\Lambda)$.
	We already know that $N_\lambda\otimes_\Lambda K=e_1N_K$, and it is clear from the above arguments that $(\cdot,\cdot)$ restricts to a perfect pairing $e_1M_K \times e_1N_K \to K$. We thus only have to check that the inverse image of $\Hom_\Lambda(N_\lambda,\Lambda)$ under $e_1M_K \simeq \Hom_\Lambda(e_1N_K,K)$, say $M'$, is equal to $M^\lambda$. But we have $M'=(N_\lambda)^*\cap e_1M_K$ by definition, so we conclude that $M'=(M^\lambda \oplus e_2M_K)\cap e_1M_K=M_\lambda$. 
\end{proof}

\begin{corAppendix}\label{coro:appendix_generator_congruence_module}
	Let $m_0$ and $n_0$ be generators of $M_\lambda$ and $N_\lambda$ over $\Lambda$ respectively (see Lemma \ref{lem:appendix_direct_summand}). The $\Lambda$-linear map 
	\[\varphi \colon \left\{ \begin{array}{rcl}
		M^\lambda &\to& \Lambda \\ m & \mapsto & (m,n_0)
	\end{array} \right.\]
is an isomorphism. Consequently, $C_0(\lambda)\simeq C_0^M(\lambda) \simeq \Lambda/(m_0,n_0)\Lambda$.
\end{corAppendix}
\begin{proof}
The proof of Proposition \ref{prop:appendix_orthogonality} shows that $(\cdot,\cdot)$ restricts to a non-degenerate pairing $M^\lambda \times N_\lambda \to \Lambda$, so $\varphi$ is well-defined and injective. Again by Proposition \ref{prop:appendix_orthogonality}, the element of $\Hom_\Lambda(N_\lambda,\Lambda)$ sending $n_0$ to $1$ is of the form $(m,\cdot)$ for some $m\in M^\lambda$, and this element thus satisfies $\varphi(m)=1$. This shows that $\varphi$ is an isomorphism. It is then clear that $\varphi$ induces an isomorphism $M^\lambda/M_\lambda \simeq \Lambda/\varphi(m_0)\Lambda=\Lambda/(m_0,n_0)\Lambda$.
\end{proof}

\bibliographystyle{amsalpha}
\bibliography{bib}

\end{document}